\documentclass[a4paper,11pt]{amsart}
\usepackage[dvips]{graphics}
\usepackage[matrix,arrow,tips,curve]{xy}
\usepackage[english]{babel}
\usepackage[leqno]{amsmath}
\usepackage{amssymb,amsthm}
\usepackage{amscd}
\usepackage{enumerate}

\input{xy}
\xyoption{all}

\newtheorem{thm}{Theorem}[section]
\newtheorem{lemma}[thm]{Lemma}

\newtheorem{prop}[thm]{Proposition}

\newtheorem{cor}[thm]{Corollary}

\newtheorem{lemmadef}[thm]{Lemma - Definition}

\theoremstyle{remark}
\newtheorem{remark}[thm]{Remark}

\theoremstyle{definition}
\newtheorem{defi}[thm]{Definition}

\newtheorem{example}[thm]{Example}

\newcommand{\la}{\longrightarrow}

\newcommand{\ov}{\overline}

\newcommand{\Div}{\operatorname{Div}}

\newcommand{\supp}{\operatorname{Supp}}

\newcommand{\Pic}{\operatorname{Pic}}

\newcommand{\Jac}{\operatorname{Jac}}

\newcommand{\mdeg}{\underline{\operatorname{deg}}}
\newcommand{\val}{\operatorname{val}}

\def\L{\mathcal L}

\def\X{\mathcal X}

\newcommand{\PP}{\mathbb{P}}
\newcommand{\Z}{\mathbb{Z}}
\newcommand{\R}{\mathbb{R}}

\def\ZZ{\mathcal Z}

\def\mo{\underline{0}}

\def\mP{\underline{P}}

\newcommand{\sing}{X_{\text{sing}}}

\newcommand{\Ma}{M^{\rm{{alg}}}}
\newcommand{\Mt}{M^{\rm{{trop}}}}

\newcommand{\Mgt}{{M_g^{\rm trop}}}

\newcommand{\Mgb}{\ov{M_g}}
\newcommand{\Mgrd}{M_{g,d}^r}
\newcommand{\Mgrdb}{\ov{M_{g,d}^r}}

\newcommand{\hG}{\widehat{G}}

\newcommand{\hv}{\widehat{v}}
\newcommand{\hu}{\widehat{u}}
\newcommand{\hw}{\widehat{w}}

\newcommand{\hT}{\widehat{T}}
\newcommand{\hX}{\widehat{X}}
\newcommand{\hY}{\widehat{Y}}

\newcommand{\ooG}{\overline{G}}

\newcommand{\oov}{\overline{v}}

\newcommand{\ooT}{\overline{T}}

\newcommand{\oophi}{\overline{\phi}}
\newcommand{\ooiota}{\overline{\iota}}
\newcommand{\oow}{\overline{w}}

\newcommand{\he}{\widehat{e}}

\newcommand{\hphi}{\widehat{\phi}}
\newcommand{\hiota}{\widehat{\iota}}
\newcommand{\Ev}{E^{\rm{ver}}}
\newcommand{\Eh}{E^{\rm{hor}}}

\begin{document}
\title{ Gonality of algebraic curves and graphs}

\author{ Lucia Caporaso}
 \address{Dipartimento di Matematica e Fisica,
 Universit\`a Roma Tre, 
 Largo San Leonardo Murialdo 1,
 00146 Roma (Italy)}
 \email{caporaso@mat.uniroma3.it}
\maketitle
\begin{center}



\begin{abstract}
We define
$d$-gonal weighted  graphs using  ``harmonic indexed" morphisms, and 
  prove
that a combinatorial locus of $\Mgb$   contains a $d$-gonal curve if the corresponding graph is $d$-gonal and of Hurwitz type. 
Conversely the dual graph of a $d$-gonal stable curve  is equivalent to a $d$-gonal graph of Hurwitz type.
The  hyperelliptic case   is studied in details.
For  $r\geq 1$, we show that 
  the dual graph of a $(d,r)$-gonal stable is  the underlying graph of a tropical curve   admitting a degree-$d$ divisor   of
rank at least $r$. 
\end{abstract}

\end{center}

\tableofcontents

\section{Introduction and Preliminaries}
 \subsection{Introduction}
 In this paper we study the interplay between the theory of linear series on algebraic curves, and the    theory of  linear series on graphs.
 
 A smooth   curve $C$ is $d$-gonal if it admits a linear series of degree $d$ and rank  $1$; more generally, $C$ is $(d,r)$-gonal if it admits a linear series of degree $d$ and rank $r$.
 A stable, or  singular,  curve is defined to be
  $(d,r)$-gonal,
 if it is the specialization of a family of smooth  $(d,r)$-gonal curves.
This rather unwieldy definition is due to the fact that
  the divisor theory of singular curves is quite complex; for example,   every reducible curve admits infinitely many divisors of degree $d$ and rank $r$, for every $d$ and $r\geq 0$.
Moreover characterizing $(d,r)$-gonal curves   is a well known difficult problem.

On the other hand, the moduli space of Deligne-Mumford stable curves, $\Mgb$, has a natural stratification into ``combinatorial" loci,
 parametrizing  curves having a certain  weighted graph as dual graph. 
It is thus natural to ask whether the existence in a combinatorial locus of a $(d,r)$-gonal curve
can be detected uniquely from the corresponding graph and its divisor theory. 

In fact, in recent times a   theory for divisors on   graphs has been 
set-up and 
developed in a purely combinatorial way, revealing some  remarkable analogies with the algebro-geometric case; see
\cite{BdlHN}, \cite{BN}, \cite{BNRR}  for example. One of the goals of this paper is to contribute to this developement; we give a new definition for morphisms between graphs, which we call {\it indexed morphisms}, and then introduce {\it harmonic indexed morphisms}. Our definition is inspired by the theory of admissible coverings developed by J. Harris and D. Mumford in \cite{HM}, and  generalizes  the combinatorial definition of  {\it harmonic  morphisms} given   by M. Baker, S. Norine and  H. Urakawa in \cite{BN} and  \cite{U} for weightless graphs; this is why we use the word ``harmonic". 
Harmonic indexed  morphisms   have a well defined degree, and satisfy the Riemann-Hurwitz formula with an effective ramification divisors.

We say that a graph is $d$-gonal if it admits a non-degenerate harmonic indexed morphism, $\phi$,  of degree $d$ to a tree; furthermore
we say that it is of Hurwitz type if      the Hurwitz existence problem  naturally associated to $\phi$   has a positive solution;  see  
Definition~\ref{imdef} for details. In particular, if $d\leq 3$ every $d$-gonal graph is of Hurwitz type. Then we prove the following:

\begin{thm}
\label{mainstable}
 If $(G,w)$ is a $d$-gonal  weighted  stable graph of Hurwitz type, there exists a (stable) $d$-gonal curve whose dual graph is $(G,w)$.
 Conversely, the dual graph of a stable $d$-gonal curve is equivalent to a $d$-gonal graph of Hurwitz type.
\end{thm}

This Theorem follows immediatly from the more general Theorem~\ref{main}, whose  proof combines the theory of  admissible coverings with properties of   harmonic indexed  morphisms.
 
Next, for all $r\geq 1$ we prove Theorem~\ref{mainconv}, which, in particular,  states that
{\it the dual graph of a  $(d,r)$-gonal curve 
admits a refinement  admitting a divisor of rank $r$ and degree $d$.}
 
 The proof of this theorem uses different methods than the previous one:    the   theory of stable curves,
 and a generalization, from \cite{AC},   of  Baker's specialization lemma  \cite[Lemma 2.8]{bakersp}.

Testing whether a graph admits a divisor of given degree and rank 
 involves only a finite number of steps, and
 can be  done by a computer; hence Theorem~\ref{mainconv} 
 yields a handy necessary condition for a curve to be $(d,r)$-gonal.

This theorem has  also consequences on tropical curves. In fact 
  the moduli space of tropical curves of genus $g$, $\Mgt$, has a  stratification 
 indexed by stable weighted graphs  exactly as $\Mgb$. Using our results we obtain  that 
 if a combinatorial stratum of $\Mgb$ contains a
 $(d,r)$-gonal curve, so does the corresponding stratum of $\Mgt$; see Subsection~\ref{tropsec} for more details. The connections between the divisor theories of algebraic and tropical curves
 have been object of much interest in recent years; in fact some closely related  issues are currently being investigated, under a completely different perspective, in a joint project of  O. Amini, M. Baker, E. Brugall\'e and J. Rabinoff. 
 We refer also to  \cite{BPR}, 
  \cite{BMV}, \cite{Chbk},  \cite{chan}   and \cite{LPP} 
  for some recent work on the relation between algebraic and  tropical geometry.
 
The paper is organized in four sections; the first recalls   definitions and results from algebraic geometry and graph theory needed in the sequel, mostly from  \cite{HM}, \cite{gac}, \cite{BNRR} and \cite{AC}.
In Section~\ref{mainsec}  we study the case $r=1$ and prove
Theorem~\ref{main} (and Theorem~\ref{mainstable}). The next section studies the case $r\geq 1$
and extends the analysis to  tropical curves; the main result here is  Theorem~\ref{mainconv}.
In Section~\ref{hypsec} we concentrate on the hyperelliptic case, and develop the basic theory by extending some of the results of \cite{BN}.  It turns out that for this case the analogies between the algebraic and the combinatorial setting are stronger; see Theorem~\ref{bridgethm}. 
 
 I wish to thank   M. Baker,  E. Brugall\'e, M. Chan, R Guralnick, and F. Viviani for enlightening discussions related to the topics in this paper. I am grateful to S. Payne for pointing out an error in the first version of Theorem~\ref{main}.
  
  \subsection{Graphs and dual graphs of curves}
\label{gcdef}
Details about the forthcoming topics may be found in  \cite{gac}  and \cite{Chbk}.

 Unless we specify otherwise, by the word  ``curve" we mean reduced, projective algebraic variety of dimension one over the field of complex numbers; we always assume that our curves have   at most nodes as singularities. 
The genus of a curve is the arithmetic genus.

The graphs we consider, usually denoted by a ``$G$" with some decorations,
are connected  graphs (no metric) admitting loops and multiple edges, unless differently stated.
For the reader's convenience we recall some basic terminology from graph theory. 
Our  conventions
 are chosen to fit both the combinatorial and algebro-geometric set up. 
 For a graph $G$ we denote by $V(G)$ the set of its vertices, by $E(G)$ the set of its edges and by $H(G)$ the set of its half-edges. The set of half-edges comes with a fixed-point-free involution   
whose orbits, written $\{h,  \ov{h}\}$, bijectively correspond to $E(G)$,
and with a surjective {\it endpoint} map $\epsilon:H(G)\to V(G)$.
For $e\in E(G)$ 
corresponding to the half-edges $h,  \ov{h}$ we often write $e=[h,  \ov{h}]$.

A {\it loop-edge} is an edge $e=[h,  \ov{h}]$ such that $\epsilon(h)=\epsilon(\ov{h})$.

A {\it leaf} is a pair, $(v,e)$,
of a vertex and  an edge, where $e$ is  not a loop-edge and is the unique edge adjacent to $v$.
We say that $e$ is a {\it leaf-edge} and $v$ is a {\it leaf-vertex}.

A {\it bridge} is an edge  $e$ such that $G\smallsetminus e$ is disconnected.

Let $v\in V(G)$; we denote by $E_v(G)\subset E(G)$, respectively by $H_v(G)\subset H(G)$, the set of edges,
 resp. of half-edges,
adjacent to $v$.

 In some cases we will need to consider graphs endowed with legs,
then we will explicitly speak about {\it graphs with legs}.  A leg of a graph $G$ is a one-dimensional open simplex having exactly one endpoint $v\in V(G)$.  
We denote by $L(G)$ the set of  legs of $G$, and by $L_v(G)$ the set of legs having $v$ as endpoint.

The {\it valency},   $\val (v)$, of a vertex $v\in V(G)$ is defined as follows
\begin{equation}
\val (v):=|H_v(G)|+|L_v(G)|.
\end{equation}

Let now $X$ be a curve (having at most nodes as singularities), and let $G_X$ be its 
 so-called dual graph. So, 
the vertices of $G_X$ correspond to the irreducible components of $X$, and we write
  $X=\cup_{v\in V(G_X)}C_v$ with $C_v$ irreducible curve.
The edges
of $G_X$ correspond to the nodes of $X$, 
and  we   denote the set of nodes of $X$ by $\sing=\{N_e,\  e\in E(G_X)\}$.
The endpoints of the edge $e$ correspond to the  components of $X$ glued at the node $N_e$.
Finally, the set of half-edges $H(G_X)$ is identified with the set of points
of the normalization of $X$ lying over    the nodes, so that a pair
$\{h, \ov{h}\}\subset H(G_X)$ corresponding to the edge $e\in E(G_X)$
is identified with a pair of points $p_h, p_{\ov{h}}$ on the normalization of $X$
in such a way that, denoting by $v,\ov{v}$ the endpoints of $e$,
with $h$ adjacent to $v$ and $\ov{h}$ adjacent to $\ov{v}$,
we have that $p_h$ lies on the normalization of $C_v$ and  $p_{\ov{h}}$
on the normalization of  $C_{\ov{v}}$.
This yields a handy  description of $X$:
\begin{equation}
\label{decompX}
X=\frac{\sqcup_{v\in V(G_X)}C_v^{\nu}}{\{ p_h=p_{\ov{h}},\  \forall h\in H(G_X)\}}
\end{equation}
where 
$C_v^{\nu}$ denotes the normalization of $C_v$.

Next,
let $(X; x_1,\ldots, x_b)$ be a   {\it pointed} curve, i.e. $X$ is a curve and $x_1,\ldots, x_b$
are   nonsingular points of $X$. To  $(X; x_1,\ldots, x_b)$
  we associate a graph with legs, 
  written
  $$
  G_{(X; x_1,\ldots, x_b)},
  $$
  by adding to the dual graph  $G_X$ described above
one leg $\ell_i$ for each marked point $x_i$, so that the endpoint of $\ell_i$
is the vertex  $v$ such that  $x_i\in C_v$. 
 
 \
  
A  {\it weighted graph} is a pair $(G,w)$ where  $G$ is a graph (possibly with legs) and $w$ a weight function $w:V(G)\to \Z_{\geq 0}$.
 The genus of a weighted graph is  
 $$
 g_{(G,w)}:=b_1(G)+\sum_{v\in V(G)}w(v).
 $$
 
A {\it tree} is a connected graph of genus zero (hence weights equal zero).

A weighted graph $(G,w)$  with legs   is {\it stable} (respectively {\it semistable}),
 if for every vertex $v$ we have
 $$w(v)+\val(v)\geq 3 \  \  \  \  {\text{(resp.}} \geq  2{\text{).}}$$

 \begin{defi}
\label{stabgr}
Let $(G,w)$ be a weighted graph of genus at least $2$. Its {\it stabilization}   is the stable graph 
obtained
obtained by removing from $(G,w)$  all leaves $(v,e)$ such that $w(v)=0$ and all
2-valent vertices of weight zero (see below). 
We say that two graphs are {\it (stably) equivalent} if they have the same stabilization.
\end{defi}
The stabilization does not change the genus. 

As in the previos definition, we shall often speak about graphs obtained by ``removing"   a 2-valent vertex, $v$,
from a given graph, $G$. By this we mean that after removing $v$, the topological space
of the so-obtained graph  is the same as that of $G$, but the sets of  vertices and edges are different.
The operation opposite  to removing a 2-valent vertex is that of ``inserting" a vertex (necessarily 2-valent) in 
the interior of an edge.

 A {\it refinement} of  a  weighted graph $(G,w)$ is a weighted graph  
   obtained by inserting some weight zero vertices in   the interior of some edges of $G$.

 Let now $X$ be a curve as before.
The  {\it (weighted) dual graph} of $X$ is the   weighted graph
 $(G_X, w_X)$, with $G_X$ as defined above, and for $v\in V(G_X)$
 the value $w_X(v)$ is equal to the genus of the normalization of $C_v$.
 
 It is easy to see that the genus of $X$ is equal to the genus of $(G_X, w_X)$. 
 
 The (weighted) dual graph of a pointed curve $(X; x_1,\ldots, x_b)$ is the graph with legs $(G_{(X; x_1,\ldots, x_b)}, w_X)$.

 \begin{remark}
 A pointed curve $(X; x_1,\ldots, x_b)$ is {\it stable}, or {\it semistable},   if and only if so is $(G_{(X; x_1,\ldots, x_b)}, w_X)$.

A curve $X$ is {\it rational} (i.e. it has genus zero) if and only if $(G_X,w_X)$ is a tree.
\end{remark}

 \begin{remark}
 \label{grstab}
Let $X$ be a curve of genus $\geq 2$ and $(G_X,w_X)$ its dual graph. 
There exists a unique stable curve $X^s$ of genus $g$ 
with a surjective map $\sigma: X\to X^s$,  such that $\sigma$ is birational away from some smooth rational components that get contracted
to a point.   $X^s$  is called the {\it stabilization} of $X$.
The dual graph of  $X^s$ is the stabilization of $(G_X,w_X)$; see Definition~\ref{stabgr}.
\end{remark}

For a stable  graph $(G, w)$ of genus $g$, we denote by
$\Ma(G, w)\subset \Mgb$ the locus of curves whose dual graph is $(G, w)$,
and we refer to it as a {\it combinatorial locus} of $\Mgb$
(the superscript ``alg" stands for algebraic, versus tropical,  see Subsection~\ref{tropsec}). Of course, we have
\begin{equation}
\label{algpart}
\Mgb=\bigsqcup_{(G, w)   {\text{ stable, genus }   g}}\Ma(G, w).
\end{equation}

 \subsection{Admissible coverings}
 \label{acssec}
 Details about this subsection may be found in   \cite{HM}, \cite{HMo} and \cite{gac}.
Let $\Mgb$ be the moduli space of stable curves of genus $g\geq 2$ and $M_g\subset \Mgb$
its open subset parametrizing smooth curves.
We denote by $ \Mgrdb$ the closure in  $\Mgb$ of the locus, $\Mgrd$, of smooth curves 
admitting a divisor of degree $d$ and rank $r$;
  in symbols:
\begin{equation}
\label{Mgrd}
\Mgrd:=\{[X]\in M_g: W^r_d(X)\neq \emptyset\}
\end{equation}
where $W^r_d(X)$ is the set of linear equivalence classes of divisors $D$ on $X$ such that $h^0(X,D)\geq r+1$.

The case of hyperelliptic curves, $r=1$ and $d=2$, has traditionally a simpler notation:
one denotes by $H_g\subset M_g$ the locus of hyperelliptic curves and by $\overline{H_g}$
its closure in $\Mgb$.
So, $\overline{H_g}=\ov{M^1_{g,2}}$.

\begin{defi}
\label{dgcdef}
Let $X$ be a connected  curve of genus $g\geq 2$.

If  $X$ is stable, then  $X$ is {\it hyperelliptic} if $[X]\in \overline{H_g}$; more generally
$X$ is {\it $(d,r)$-gonal}, respectively  {\it $d$-gonal}, if $[X]\in  \Mgrdb$, resp. 
if $[X]\in \ov{M^1_{d,g}}$.

If  $X$ is arbitrary, we say  $X$
 is hyperelliptic, $(d,r)$-gonal, or $d$-gonal if so is its stabilization.
 
 A  connected  curve of genus $g\leq 1$ is $d$-gonal for all $d\geq 2$. 
 \end{defi}

 We recall  the  definition of admissible covering, due to J. Harris and D. Mumford
 \cite[Sect. 4]{HM}, and introduce some useful generalizations.
 
\begin{defi}
\label{acdef}
Let  $Y$ be a connected  nodal  curve of genus zero, and $y_1,\ldots, y_b$ be   nonsingular points of $Y$; let 
$X$ be a connected nodal curve.
\begin{enumerate}[{\bf(A)}]
\item
\label{ac3}
A  {\it  covering}  (of $Y$)  is a regular map $\alpha:X\to Y$   such that the following conditions hold:
\begin{enumerate}
\item
$\alpha^{-1}(Y_{\rm sing})=\sing.$
\item
\label{ac3b}
$\alpha$ is unramified away from $\sing$ and away from
$y_1,\ldots, y_b$.
\item
\label{ac3c}
 $\alpha$  has  simple ramification  (i.e. a single point with ramification index  equals 2) 
 over   $y_1, \ldots, y_b$.
\item
\label{ac33}
For every $N\in \sing$ the ramification indices of $\alpha$ at the two branches of $N$ coincide.
\end{enumerate}
\item
\label{ac1}
A covering is called {\it semi-admissible} (resp. {\it admissible})
if the pointed curve 
$(Y; y_1,\ldots, y_b)$
is semistable (resp.   stable),
 i.e. 
for every irreducible component $D$ of $Y$ we have
\begin{equation}
\label{eqst}
|D\cap \ov{Y\smallsetminus D}|+|D\cap \{y_1,\ldots, y_b\}|\geq 2  \  \  \  \  (\text{resp.}  \geq 3).
\end{equation}
\end{enumerate}
\end{defi}
We shall write $\alpha:X\to (Y; y_1,\ldots, y_b)$ for a  covering as above,
and sometimes  just  $\alpha:X\to Y$. 
In fact the definition of a  covering (without its being semi-admissible) does not need the points $y_1,\ldots, y_b$, as conditions \eqref{ac3b}  and \eqref{ac3c} may be replaced by imposing that $\alpha$ has ordinary ramification away from $\sing$.
The following are simple consequences of the definition.
\begin{remark}
\label{acrk} Let $\alpha:X\to Y$ be a   covering.
\begin{enumerate}[(A)]
\item
\label{acdeg}
There exists an integer $d$ such that for every irreducible component 
$D\subset Y$  
the degree of $\alpha_{|D}:\alpha^{-1}(D)\to D$ is $d$. We say that $d$ is the degree of $\alpha$.
\item
\label{acnoloop}
Every irreducible component of $X$ is nonsingular. 
\item
\label{acdeg2}
If $\alpha$ is admissible of degree 2, then $X$ is semistable.
\end{enumerate}
\end{remark}
In \cite{HM} the authors construct the moduli space $\ov{H_{d,b}}$ for admissible coverings, as a projective irreducible variety compactifiying the  Hurwitz scheme
(parametrizing admissible coverings having smooth range and target), and show that it has a natural morphism
\begin{equation}
\label{acmap}
\ov{H_{d,b}}\la \Mgb;\quad \quad [\alpha:X\to Y]\mapsto [X^s]
\end{equation}
where $X^s$ is the stabilization of $X$ and $g$ is its genus, so that $b=2d+2g-2$.
For example, if $d=2$ we have	
$
\ov{H_{2,2g+2}}\la \Mgb
$.

Moreover, the image of $\ov{H_{2,2g+2}}$ coincides with the locus  of  hyperelliptic stable curves,  $\ov{H_g}$, and
more generally
 the image of (\ref{acmap}) is the closure in $\Mgb$ of the locus of $d$-gonal curves, here denoted by $\ov{M_{g,d}^1}$.
 
The description of an explicit admissible covering  is in Example~\ref{acex}.

\subsection
{Divisors on  graphs}
\label{spgp}
For any graph $G$,  or any weighted graph $(G,w)$, its divisor group,
$\Div G$,  or $\Div(G,w)$, is defined as the free abelian group generated by the vertices of $G$.
We use the following notation for a  divisor $D$ on $(G,w)$
\begin{equation}
\label{notdiv}
D=\sum_{v\in V(G)}D(v)v
\end{equation}
where $D(v)\in \Z$.
For loopless and weightless graphs
 we use the divisor theory developed in \cite{BNRR}.
 If $G$   is a weighted graph with loops, we extend this theory as in   \cite{AC}. We begin with a definition. 
 
\begin{defi}
 \label{varG}
 Let $(G,w)$ be a weighted graph.

 We denote by  $G^0$ the loopless  graph obtained from $G$ by inserting a
vertex in the interior of every loop-edge,
and  by $(G^0,w^0)$ the  weighted graph such that
  $w^0$ extends $w$  and is  equal to zero on all vertices in $V(G^0)\smallsetminus V(G)$.

We denote by $G^w$ the  weightless, loopless graph obtained from $G^0$ by adding $w(v)$ loops based at $v$ for every $v\in V(G)$
and then inserting a
vertex in the interior of every loop-edge.
\end{defi}
 
Notice that $(G,w)$,   $(G^0,w^0)$ and $G^w$ have the same genus, and that $(G^0)^{w^0}=G^w$.
  
 For every $D\in \Div (G,w)$ its rank, $r_{(G,w)}(D)$,   is set equal to
$ r_{G^w}(D)$.
 Linearly equivalent divisors have the same rank.
A weighted graph $(G,w)$ of genus $g$ has a canonical divisor
 $K_{(G,w)}=\sum_{v\in V(G)}(2w(v)-2+\val(v))v$  of degree $2g-2$ such that the following Riemann-Roch formula holds
 \cite[Thm. 3.8]{AC}
$$
r_{(G,w)}(D)-r_{(G,w)}(K_{(G,w)}-D)=\deg D - g +1.
$$
\begin{remark}
\label{g01}
A consequence of the Riemann-Roch formula is the fact that if $g\leq 1$ then
for any divisor $D$ of degree $d\geq 0$  we have $r_{(G,w)}(D)=d-g$.
\end{remark}

For a weighted graph $(G,w)$ we denote by $ \Jac^d(G,w)$ the set of  linear equivalence classes  of degree-$d$ divisors, and set  
$$W^r_d(G,w):=\{[D]\in \Jac^d(G,w): r_{(G,w)}(D)\geq r\}.
$$
\begin{defi}
\label{ddef}
We say that a graph $(G,w)$ is {\it divisorially $d$-gonal}  if it admits a divisor of degree $d$ and rank at  least $1$, that is if $W^1_d(G,w)\neq \emptyset.$

A {\it hyperelliptic} graph is a divisorially $2$-gonal graph.
\end{defi}
\begin{example}
\label{binex}
Consider the following graph $G$ with $n\geq 2$.

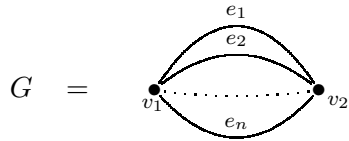
\begin{figure}[h]
\begin{equation*}
\xymatrix@=.5pc{
 &\\
G&=  &&*{\bullet}
  \ar @{.} @/_.2pc/[rrrrr]_(0.01){v_1} \ar @{-} @/_1.5pc/[rrrrr]^{e_n} _(1){v_2}\ar@{-} @/^1.1pc/[rrrrr]^{e_2}\ar@{-} @/^2pc/[rrrrr]^{e_1}
&&&&& *{\bullet} &&&&
\\
 &\\
}
\end{equation*}
\end{figure}
$G$ is obviously hyperelliptic, as $r_G(v_1+v_2)= 1$.
Notice also that  
$$r_G(2v_1)=
\begin{cases}
1 & \text{ if $n=2$}\\
0 & \text{ if $n\geq 3$.}
\end{cases}
$$
Now fix on $G$   the  weight function  
 given by $w(v_1)=0$ and $w(v_2)=1$.
Here is the picture of $G^w$ (drawing weight-zero vertices by a ``$\circ$")

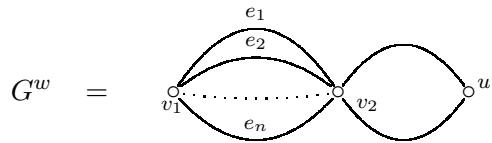
\begin{figure}[h]
\begin{equation*}
\xymatrix@=.5pc{
 &\\
G^w &=&&*{\circ}
  \ar @{.} @/_.2pc/[rrrrr]_(0.01){v_1} \ar @{-} @/_1.5pc/[rrrrr]^{e_n} _(1){\  v_2}\ar@{-} @/^1.1pc/[rrrrr]^{e_2}\ar@{-} @/^2pc/[rrrrr]^{e_1}
&&&&& *{\circ}  \ar @{-} @/_1.5pc/[rrrr] \ar @{-} @/^1.5pc/[rrrr] ^(1){u} &&&& *{\circ}&&&
\\
 &\\
}
\end{equation*}
\end{figure}
We have $r_{(G,w)}(v_1+v_2)=r_{(G,w)}(u+v_1)=r_{(G,w)}(u+v_2)=0$ for every $n\geq 2$.
On the other hand 
$$r_{(G,w)}(2v_1)=
\begin{cases}
1 & \text{ if $n= 2$}\\
0 & \text{ if $n\geq 3$}
\end{cases}
$$
and the same holds for $2v_2\sim 2u$.
Therefore  $(G,w)$  is hyperelliptic if and only if $n=2$ (in fact $n\leq 2$).
This example is   generalized in Corollary~\ref{binary}
\end{example}

\section{Admissible coverings and harmonic morphisms}
\label{mainsec}
 \subsection
{Harmonic morphisms of graphs}
\label{hyploop}
Let $\phi:G\to G'$ be a morphism; we denote by 
$\phi_V:V(G)\to V(G')$ the map induced by $\phi$ on the vertices.
$\phi$ is a  {\it homomorphism} if  $\phi(E(G))\subset E(G')$;
in  this  case we denote by $\phi_E:E(G)\to E(G')$ and by $\phi_H:H(G)\to H(G')$ the induced maps on edges and half-edges.
A morphism between weighted  graphs $(G,w)$ and $(G',w')$ is defined as a morphism of the underlying graphs, so we write either $G\to G'$ or $(G,w)\to (G',w')$ depending on the situation.

In the next definition, extending the one in \cite[Subsect. 2.1]{BN}, we introduce some extra structure on   morphisms between graphs.

\begin{defi}
\label{imdef}
Let $(G,w)$ and $(G',w')$ be loopless weighted graphs.
\begin{enumerate}[{\bf(A)}]
\item
An {\it indexed morphism} is a morphism  $\phi:(G,w)\to (G',w')$ enriched by the assignment,
for every $e\in E(G)$, of a non-negative integer, the  {\it index} of $\phi$ at $e$,
written $r_{\phi}(e)$, 
such that $r_{\phi}(e)=0$ if and only if $\phi(e)$ is a point.
An indexed morphism is {\it simple} if $r_{\phi}(e)\leq 1$ for every $e\in E(G)$.
Let $e=[h, \ov{h}]$ with $h, \ov{h}\in H(G)$; we set
$r_{\phi}(h)=r_{\phi}(\ov{h})=r_{\phi}(e)$.
\item
An indexed morphism is {\it pseudo-harmonic}
if for every $v\in V(G)$ there exists a number, $m_{\phi}(v)$,
such that for every $e'\in E_{\phi_V(v)}(G')$  (and, redundantly for convenience, every $h'\in H_{\phi_V(v)}(G')$)  we have
\begin{equation}
\label{imeq}
m_{\phi}(v)=\sum_{e\in E_v(G): \phi(e)=e'}r_{\phi}(e)=\sum_{h\in H_v(G): \phi(h)=h'}r_{\phi}(h). 
\end{equation}
\item
A pseudo-harmonic indexed morphism   is {\it non-degenerate} if $m_{\phi}(v)\geq 1$ for every $v\in V(G)$.

 \item
A pseudo-harmonic indexed  morphism is {\it harmonic} if 
 for every $v\in V(G)$ we have, writing  $v'=\phi(v)$,
 \begin{equation}
\label{RH}
\sum_{e\in E_v(G)}(r_{\phi}(e)-1) \leq 2\Bigr(m_{\phi}(v)-1+w(v)-m_{\phi}(v)w'(v')\Bigl).
\end{equation}

\end{enumerate}
\end{defi}
In the sequel, 
  all graph morphisms will be indexed morphisms, hence 
we shall usually omit the word ``indexed".

For later use, let us observe that if $w'=\mo$ (i.e. $G'$ is weightless) condition \eqref{RH} simplifies as follows
\begin{equation}
\label{rest2}
\sum_{e\in E_v(G)}(r_{\phi}(e)-1) \leq 2(m_{\phi}(v)-1+w(v)). 
\end{equation}

\begin{remark}
Suppose that $\phi$ contracts a leaf-edge $e$ whose leaf-vertex $v$ has $w(v)=0$.
Then $r_{\phi}(e)=m_{\phi}(v)=0$ and condition \eqref{RH} is not satisfied on $v$. So, loosely speaking, a harmonic morphism contracts no weight-zero leaves. 
\end{remark}
\begin{remark}
\label{BNrk} {\it Relation with harmonic morphisms of \cite{BN}.}
For simple morphisms of weightless graphs our definition of harmonic morphism coincides with
the one  given   in \cite[Sec. 2.1]{BN}    for morphisms which contract 
no leaves. 
Indeed, it is clear that any simple pseudo-harmonic morphism is harmonic in the sense of \cite{BN}. Conversely, a harmonic morphism in the sense of \cite{BN} satisfies \eqref{rest2}
(with $w(v)=0$)
if and only if $\phi$ contracts no leaves; see the previous remark.
\end{remark}

 \begin{lemmadef}
Let $\phi:(G,w)\to (G',w')$ be a pseudo-harmonic    morphism.
Then for every $e'\in E(G')$ and $v'\in V(G')$ we can define the degree of $\phi$ as follows 
\begin{equation}
\label{deghar}
\deg \phi = \sum_{e\in E (G): \phi(e)=e'}r_{\phi}(e) =\sum_{v\in \phi^{-1}(v')}m_{\phi}(v) 
\end{equation}
(i.e. the above summations are independent of the choice of $e'$ and $v'$).
\end{lemmadef}
\begin{proof}
Trivial extension of the proof of \cite[Lm. 2.2 and Lm. 2.3]{BN}.
\end{proof}
Let $\phi:(G,w)\to (G',w')$ be a pseudo-harmonic    morphism.
As in \cite[Subs. 2.3]{BN} we define a pull-back homomorphism 
$\phi^*:\Div(G',w')\to \Div(G,w)$  as follows: for every $v'\in V(G')$
\begin{equation}
\label{deg1}
\phi^*v'=\sum_{v\in \phi^{-1}(v')}m_{\phi}(v)v 
\end{equation}
 and we extend this linearly to all of $\Div(G',w')$. 
By \eqref{deghar} we have
 \begin{equation}
 \label{degdiv}
\deg D=\deg\phi\deg D'.
\end{equation}
 For a pseudo-harmonic morphism $\phi$   the {\it ramification divisor} $R_{\phi}$ is defined
 as follows.
\begin{equation}
\label{Ramdef}
R_{\phi}=\sum_{v\in V(G)} \Bigr(2\bigr(m_{\phi}(v)-1+w(v)-m_{\phi}(v)w'(v')\bigl)-\sum_{e\in E_v(G)}(r_{\phi}(e)-1)  \Bigl)v.
\end{equation}
The next result, generalizing the analog in \cite{BN}, implies that harmonic morphisms are characterized, among pseudo-harmonic morphisms, by   a Riemann-Hurwitz formula with   effective ramification divisor.
 \begin{prop}[Riemann-Hurwitz]
Let $\phi:(G,w)\to (G'w')$ be a pseudo-harmonic morphism of weighted graphs of genus $g$ and $g'$ respectively.
Then
\begin{equation}
\label{RHeq}
 K_{(G,w)}=\phi^*K_{(G',w')}+R_{\phi}.
\end{equation}
  $\phi$ is harmonic if and only if $R_{\phi}\geq 0$ (equivalently  $2g -2\geq \deg \phi(2g'-2)$).
\end{prop} 
 \begin{proof}
 We write $K=K_{(G,w)}$ and $K'=K_{(G',w')}$.
For every $v\in V(G)$ we have $K(v)=2w(v)-2+\val(v)$ (notation in \eqref{notdiv}).
Hence, writing $v'=\phi(v)$, by \eqref{deg1} we have \begin{eqnarray*}
& K(v)-\phi^*K'(v)=2w(v)-2+\val(v)-m_{\phi}(v)\Bigr(2w(v')-2+\val(v')\Bigl)=\\
 \\
  &=2\Bigr(m_{\phi}(v)-1+w(v)-m_{\phi}(v)w(v')\Bigl)+\val(v)- m_{\phi}(v)\val(v').
 \end{eqnarray*}
 On the other hand by \eqref{deghar}
 $$
 \sum_{e\in E_v(G)}(r_{\phi}(e)-1)=\sum_{e\in E_v(G)}r_{\phi}(e)-\val(v)=m_{\phi}(v)\val(v')-\val(v).
 $$
 The two above identities imply 
 $K(v)-\phi^*K'(v)=R_{\phi}(v)$, so 
 \eqref{RHeq} is proved. 
 
 By definition, 
 $\phi$ is harmonic if and only if its ramification $R_{\phi}$ divisor is effective.  The   equivalence  in parenthesis follows from \eqref{degdiv}.
 \end{proof}
 \begin{remark}
\label{BNext} Other results proved in \cite{BN} for simple harmonic morphisms extend. In particular,  if $D'$ and $E'$ are linearly equivalent divisors on $(G',w')$,
their pull-backs $\phi^*D'$ and  $\phi^*E'$ under a pseudo-harmonic morphisms $\phi$  are linearly equivalent. 
\end{remark}

\subsection{The Hurwitz existence problem}
\label{HEP}
Our goal is to use harmonic morphisms   to characterize graphs that are dual graphs of $d$-gonal curves.
This brings up the  ``Hurwitz existence problem", about the existence of
  branched coverings of $\PP^1$  with prescribed  ramification profiles; to state it precisely we need some terminology.

Let $d\geq 1$ be an integer and let $\mP=\{P_1,\ldots, P_b\}$ be a set
of partitions of $d$,
so that we write  $P_i=\{r_i^1,\ldots,r_i^{n_i}\}$   with $r_i^j\in \Z_{\geq 1}$ and $\sum_{j=1}^{n_i} r_i^j =d$.

We say that $\mP$ is a {\it Hurwitz partition set}, or that $\mP$ is of Hurwitz type,  if the following condition holds.
There exist $b$ permutations
$\sigma_1,\ldots,\sigma_b\in S_d$  ($S_d$ the symmetric group)
whose product is equal to the identity,
such that 
$\sigma_i$ is the product of $n_i$ disjoint cycles of lengths
given by $P_i$,
and such that the subgroup
$<\sigma_1,\ldots,\sigma_b>$ is   transitive.

Notice that if $\mP$ is of Hurwitz type and we add to it the {\it trivial} partition $\{1,1,\ldots,1\}$,
  the resulting partition set is again of Hurwitz type.

\begin{remark}
\label{Riemrk}
By the Riemann existence theorem,   $\mP$ is a Hurwitz partition set  if and only if there exists a degree-$d$ connected covering $\alpha:C\to \PP^1$
 with $q_1,\ldots,q_b\in \PP^1$ such that $\alpha$ is unramified away from
 $q_1,\ldots,q_b$ and such that for all $i=1,\ldots, b$ we have
 $\alpha^*(q_i)=\sum _{j=1}^{n_i}r_i^jp_i^j$. The  genus $g$ of $C$ is determined  by the Riemann-Hurwitz formula:
\begin{equation}
\label{RiHu}
2g-2=-2d+\sum_{i=1}^b\sum_{j=1}^{n_i}(r_i^j-1),
\end{equation}
so that we shall also say that $\mP$ is a   {\it Hurwitz partition set
of genus $g$ and degree $d$}.
\end{remark}
 
\begin{remark}
\label{Hurk}
It is a fact that a   partition set $\mP$ satisfying \eqref{RiHu} is not necessarily of  Hurwitz type. Indeed, the so-called Hurwitz existence problem can be stated as follows:
characterize  Hurwitz partition  sets among all   $\mP$ satisfying \eqref{RiHu}.
This problem turns out to be very difficult and  is open in general. 
Easy cases in which every $\mP$  satisfying \eqref{RiHu} is of Hurwitz type are
$P_i=(2,1,\ldots,1)$ for every $i$, or
  $d\leq 3$, or  $b\leq 2$.
  
On the other hand  if $d=4$  the partition set $\mP=\{(3,1);(2,2);(2,2)\}$ 
  is not of  Hurwitz type, but the Riemann-Hurwitz formula \eqref{RiHu} holds with $g=0$; see \cite{PP} for this and other   results on the Hurwitz existence problem. \end{remark}

Let now $\phi:(G,w)\to T$ be a non-degenerate   pseudo-harmonic morphism, where $T$ is a tree; let $v\in V(G)$.
For any 
half-edge $h'\in H(T)$ in the image of some half-edge adjacent to $v$  we define, using \eqref{imeq},
a partition of $m_{\phi}(v)$:
$$
P_{h'}(\phi, v):=\{r_{\phi}(h),\  \forall h\in H_v(G): \  \phi(h)=h'\}.
$$
Now we   associate   to $v$ and $\phi$ the following  partition set:
\begin{equation}
\label{partv}
\mP(\phi,v)=\{P_{h'}(\phi,v),\  \forall h'\in \phi_H(H_v(G))\}.
\end{equation}
In the next definition we use the terminology of Remark \ref{Riemrk}.
 \begin{defi}
\label{dgdef}\begin{enumerate}[{\bf(A)}]
\item
Let $(G,w)$ be a loopless weighted graph.
 We say that $(G,w)$ is {\it  $d$-gonal} if
it admits   a non-degenerate,
degree-$d$ harmonic   morphism $\phi: (G,w) \to T$
where $T$ is a tree. 

If such a $\phi$ has the property that 
for every $v\in V(G)$ the partition set $\mP(\phi,v)$ is contained in a  Hurwitz partition set
of genus $w(v)$,
we say that $\phi$ is a  morphism  of {\it{Hurwitz type}}, and that
$(G,w)$ is a $d$-gonal graph of {\it{Hurwitz type}}.

\item Let $(G,w)$ be any graph. We say that it is  $d$-gonal, or of Hurwitz type,  if so is   $(G^0,w^0)$, with  $(G^0,w^0)$  as in Definition~\ref{varG}.
\end{enumerate}
\end{defi}
\begin{example}
\label{dleq3}
A harmonic morphism with indices at most equal to 2 is   of
Hurwitz type.   Hence
if  $d\leq 3$  a
$d$-gonal graph 
  is always of
Hurwitz type.  
 \end{example}
The following is one of the principal results of this paper, of which Theorem~\ref{mainstable} is a special case. Recall the terminology introduced in
Definition~\ref{stabgr}.
\begin{thm}
\label{main}
Let $(G,w)$ be a  $d$-gonal graph of Hurwitz type;
then there exists a $d$-gonal curve  
whose dual graph   is  $(G,w)$.

Conversely, let $X$ be a $d$-gonal curve; then its dual graph is equivalent to a $d$-gonal graph of Hurwitz type.
\end{thm}

The proof of the first part of the theorem will be given in Subsection~\ref{proofmain}. 
The converse is   easier, and will be proved earlier, in Corollary~\ref{pfpart2}.

 \subsection{The dual graph-map   of  a covering.}
\label{dualXdef}
To prove Theorem~\ref{main} we
  shall associate to any covering 
$\alpha:X\to Y$
an indexed morphism of graphs, called  
the {\it dual graph-map} of $\alpha$, and denoted  
 by 
 $$\phi_{\alpha}:(G_X,w_X)\la G_Y.$$ 
As all components of $Y$ have genus zero, we omit the weight function for $Y$. 
 We sometimes write just  $G_X\to G_Y$  for simplicity.

   We use the notation of subsection~\ref{gcdef};  denote by   $$Y=\cup_{u\in V(G_Y)}D_u$$
the irreducible component decomposition of  $Y$. 
For any $v\in V(G_X)$ we have that $\alpha(C_v)$ is an irreducible component of $Y$,
hence there is a unique $u\in V(G_Y)$
such that $\alpha(C_v)=D_u$; this defines a map
 $ 
\phi_{\alpha,V}: V(G_X)\to V(G_Y) 
$ mapping $v$ to $u$.

Next,   $E(G_X)$ and $E(G_Y)$ are  identified with the set of nodes of $X$ and  $Y$.
To define
$ 
\phi_{\alpha,E}: E(G_X)\to E(G_Y) $
let $e\in E(G_X)$; then $e$ corresponds to the node $N_e$ of $X$.
The image $\alpha(N_e)$ is a node of $Y$, corresponding to a unique edge of $G_Y$, which we set to be the image of $e$ under $\phi_{\alpha,E}$.

It is trivial to check that the pair $(\phi_{\alpha,V},\phi_{\alpha,E})$ defines a morphism of graphs,  
$\phi_{\alpha}:G_X\to G_Y$.

Let us now define the indices of $\phi_{\alpha}$.
For any $e\in E(G)$ let $N_e$ be the corresponding node
of $X$.  
By Definition~\ref{acdef}, the restriction of $\alpha$ to each of the two branches of $N_e$ has the form
$u=x^r$ and $v=y^r$ where $x$ and $y$ are local coordinate at the branches of $N_e$, and $u$, $v$ are local coordinates at the branches of $\alpha(N_e)$ (which is a node of $Y$).
We set $r_{\phi_{\alpha}}(e)=r$.

If we need to keep track of the branch points of  
$\alpha:X\to (Y; y_1,\ldots, y_b)$, we 
endow the dual graph of $Y$ with $b$ legs, in the obvious way, and write
$ 
\phi_{\alpha}: G_X\to G_{(Y; y_1,\ldots, y_b)}.
$ 
\begin{example}
 \label{acex}
{\it Dual graph-map for the admissible covering of an irreducible hyperelliptic curve.}
Let $X\in \ov{H_g}$ be an irreducible singular hyperelliptic curve. Such curves are completely characterized; we here choose    $X$ 
irreducible with  $g$ nodes, so
  that its normalization is $\PP^1$.
Let us describe an  admissible covering $\alpha:Z\to Y$ which maps to $X$ under the map (\ref{acmap}).
As we noticed in Remark~\ref{acrk}, $Z$ cannot be equal to $X$.
In fact, $Z$ is the ``blow-up" of $X$ at its $g$ nodes, so that $Z=\cup_{i=0}^gC_i$ is the union
of $g+1$ copies of $\PP^1$, with one copy, $C_0$, corresponding to the normalization of $X$,
and the remaining copies corresponding to the ``exceptional" components.
Hence $|C_i\cap C_0|=2$ and $|C_i\cap C_j|=0$ for all $i,j\neq 0$.
Now, since $X$ is hyperelliptic,   its normalization $C_0$ has a
two-to-one map to $\PP^1$, written $\alpha_0:C_0\to D_0\cong \PP^1$, such that
$\alpha_0(p_i)=\alpha_0(q_i)=t_i\in D_0$ for every pair $p_i,q_i\in C_0$ of points
lying over the $i$-th node of $X$. Let $y_0,y_1\in D_0$ be the two branch points of $\alpha_0$.

We  assume that in $X$ the component $C_0$ is glued to  $C_i$ along the pair $p_i,q_i$.
For $i\geq 1$ we pick a two-to-one map $\alpha_i:C_i\to D_i\cong \PP^1$ such that the two points
of $C_i$ glued to $X$ have the same image, $s_i$, under $\alpha_i$.
Let $y_{2i}, y_{2i+1}\in D_i$ be the two branch points of $\alpha_i$.

We define $Y$ as the following nodal curve
$ 
Y:=\sqcup_{i=0}^gD_i/_{\{t_i=s_i,\  \forall i=1,\ldots, g\}}.
$ 
Now, $(Y;  y_{2i}, y_{2i+1},\  \forall i=0,\ldots g )$ is   stable, and it is clear that the $\alpha_i$
glue to an admissible covering $\alpha:Z\to Y$.
The dual graphs and graph-map are in the following picture, where $g=3$.
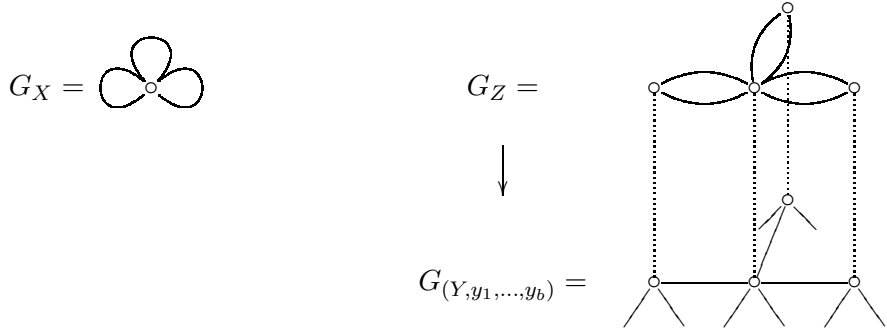
\begin{figure}[h]
\begin{equation*}
\label{figgm}
\xymatrix@=.5pc{
&&&&&&&&&&&&&&&&&&*{\circ}\\
 \\
&&G_X=&&*{\circ}\ar@{-}@(ul,dl)\ar@{-}@(ur,dr)^{}\ar@{-}@(lu,ru)&&&&&&&&G_Z=&& 
*{\circ}  \ar @{-} @/_.5pc/[rrr]\ar @{-} @/^.5pc/[rrr]^{} &&& *{\circ}\ar @{-} @/_.5pc/[rrr] _(1){}\ar@{-} @/^.5pc/[rrr]\ar@{-} @/_0.5pc/[uur]\ar@{-} @/^0.5pc/[ruu]
&&& *{\circ}&&&\\
&&& &&&&&&&&&\ar[dd]\\
 &&&&&&& &&&&&&&&&\\
 &&&&&&&&&&&&&&&&&&*{\circ}\ar@{-}[dl]\ar@{-}[dr]\ar@{.}[uuuuu]\\
 &&&&&&&&&&&&&&&&&&&&\\
 &&&&&&&&&&&&G_{(Y,y_1,\ldots,y_b)}=&& *{\circ}\ar@{-}[dl]\ar@{-}[dr] \ar@{.}[uuuuu]\ar@{-}[rrr]&&& *{\circ}\ar@{-}[dl]\ar@{-}[dr]\ar@{.}[uuuuu]\ar@{-}[rrr]\ar@{-}[uur]&&
 &*{\circ}\ar@{-}[dl]\ar@{-}[dr]\ar@{.}[uuuuu]\\
&&&&&&&&&&&&&&&&&&&&&&&&&&&\\
}
\end{equation*}
\end{figure}
\end{example}

 \begin{lemma}
 \label{acharm}
 Let  
  $\alpha:X\to Y$ be a  covering and
$\phi_{\alpha}:(G_X,w_X)\to G_Y$ the dual graph-map defined above.
Then $\phi_{\alpha}$ is a
 harmonic homomorphism of Hurwitz type. 
 
 If $\deg \alpha =2$ and  $X$   has no separating nodes, then $\phi_{\alpha}$ is simple. 
\end{lemma}
\begin{proof}
 It is clear that $G_Y$ has no loops. By Remark~\ref{acrk} (\ref{acnoloop}), every component $C_v$ of $X$ is nonsingular, hence $G_X$   has no loops.
  
Since $\alpha$ is a    covering, we have that   $\phi_{\alpha,V}$ and $\phi_{\alpha,E}$ are surjective,
and  $\phi_{\alpha}$ 
does not contract any edge of $G_X$; hence  $\phi_{\alpha}$ 
is a homomorphism. We shall abuse notation by writing $\phi_{\alpha}$ for    $\phi_{\alpha,V}$,
  $\phi_{\alpha,H}$ and $\phi_{\alpha,E}$.

Let now $v\in V(G_X)$ and $h'\in H_{\phi(v)}(G_Y)$,
so that $h'$ corresponds to a point in the image of $C_v$ via $\alpha$,
i.e. to a point in $D_{\phi(v)}\subset Y$. 
Consider the restriction of $\alpha$ to $C_v$:
$$
\alpha_{|C_v}:C_v\la D_{\phi(v)}.
$$
This is a finite morphism, and it is clear that for every $h'\in H_{\phi(v)}(G_Y)$
$$
\sum_{h\in H_v(G_X): \phi(h)=h'}r_{\phi_\alpha}(h)=\deg \alpha_{|C_v}.
$$
The right hand side above does not depend on $h'$, hence we may set
\begin{equation}
\label{pfdeg}
m_{\phi_{\alpha}}(v):=\deg \alpha_{|C_v}.
\end{equation}
Therefore $\phi_{\alpha}$ is pseudo-harmonic.
 To prove that $\phi_{\alpha}$ is harmonic we must prove that for every $v\in V(G_X)$ we have
\begin{equation}
\label{rest1}
\sum_{e\in E_v(G_X)}(r_{\phi_{\alpha}}(e)-1) \leq 2(m_{\phi_{\alpha}}(v)-1+w_X(v)). 
\end{equation}

Let $R\in \Div(C_v)$ be the ramification divisor of the map $\alpha_{|C_v}$
above. Then, by the Riemann-Hurwitz formula applied to $\alpha_{|C_v}$ we have, 
$$\deg R=2(m_{\phi_{\alpha}}(v)-1+w_X(v)).
$$
On the other hand 
the map $ \alpha_{|C_v}$ has ramification index $r_{\phi_{\alpha}}(h)$ at all $p_h\in H_v(G_X)$, hence
we must have
$$
R \  -\sum_{h\in H_v(G_X)}(r_{\phi_{\alpha}}(h)-1)p_h\geq 0
$$
from which (\ref{rest1})  follows. The fact that $\phi_{\alpha}$ is of Hurwitz type follows immediatly from Remark~\ref{Riemrk}.

Assume $\deg \alpha=2$ and $X$ free from separating nodes. We must prove the indices of $\phi$ are all equal to one, i.e.  that $\alpha_{C_v}$ does not ramify at the points $p_h$, for every $h\in H (G_X)$.
 By contradiction, suppose    $\alpha_{|C_v}$ is ramified at $p_h$; hence, as $\deg \alpha=2$,
 it is totally ramified at $p_h$, so that $\alpha^{-1}(\alpha(p_h))\cap C_v=p_h$.
 Since $\alpha$ is an admissible covering, we have exactly the same situation at the other branch of $N_e$, i.e.
 at $p _{\ov{h}}$. Therefore
 $$
 \alpha^{-1}(\alpha (N_e))=\{ N_e\}.
 $$
Now  $\alpha (N_e)$ is a node of $Y$, and hence it is a separating node. So, the above identity implies that $N_e$ is a separating node of $X$;  a contradiction.
   \end{proof}
   \begin{cor}
   \label{pfpart2}
The second part of Theorem~\ref{main} holds.
\end{cor}
\begin{proof}
Let $X$ be a $d$-gonal curve; we must prove that 
the  dual graph of $X$ is   equivalent to a $d$-gonal graph of Hurwitz type.
By hypothesis
 there exists an admissible covering $\hX\to Y$ of degree $d$ such that the stabilization of $\hX$ is the same as the stabilization of $X$; see the end of subsection~\ref{acssec}. 
Therefore the dual graph of $\hX$ is equivalent to the dual graph of $X$.
 By Lemma~\ref{acharm}   the dual graph of $\hX$ is  of Hurwitz type, hence we are done.
\end{proof}
The proof of the first part of Theorem~\ref{main}  will be based on the next Proposition, which is 
 a converse to Lemma~\ref{acharm}.

 \begin{prop}
 \label{corr} Let $(G,w)$ be a weighted  
 graph of genus $\geq 2$  and let $T$ be a tree.
Let $\phi:(G,w)\to T$ be a     harmonic homomorphism of Hurwitz type. Then there exists a   covering $\alpha:X\to Y$  whose dual graph map is $\phi$.\end{prop}

\begin{proof}
As $\phi$ is harmonic,
for every $v\in V(G)$ condition \eqref{rest2} holds.

 We will abuse notation and write $\phi$ also for the maps 
$V(G)\to V(T)$, $H(G)\to H(T) $ and $E(G)\to E(T) $ induced by $\phi$ . 
We begin by constructing  two curves  $X$  and $Y$ whose dual graphs are $(G,w)$ and $T$.

For every $u\in V(T)$ we pick a pointed  curve $(D_u,Q_u)$
with  $D_u\cong \PP^1$, 
and such that the (distinct) points in $Q_u$ are indexed by the  half-edges    adjacent to $u$:
$$
Q_u
=\{q_h,\  \forall h\in H_u(T)\}. 
$$

We have an obvious 
identification
 $  \cup _{u\in V(T)}Q_u= H(T).$  
To glue the curves $D_u$ to a connected nodal curve $Y$ we proceed as in 
\ref{dualXdef}, getting 
$$
Y=\frac{\sqcup_{u\in V(T)}D_u}{\{q_h=q_{\ov{h}},\  \forall h\in H(T)\}}.
$$
By construction, $T$ is the dual graph of $Y$.
 
 Now to construct $X$ we begin by finding its irreducible components $C_v$ with their gluing point sets
 $P_v$. Pick $v\in V(G)$ and $u=\phi(v)\in V(T)$. By hypothesis, $m_{\phi}(v)\geq 1$;
we claim that there exists a morphism from a smooth curve $C_v$ of genus $w(v)$ to $D_u$
\begin{equation}
\label{av1}
\alpha_v:C_v\la D_u
\end{equation}
of degree equal to $m_{\phi}(v)$ such that for every $h'\in H_u(T)$ the pull-back of the divisor $q_{h'}$ has the form
$$
\alpha_v^*q_{h'}=\sum_{\phi_H(h)=h'}r_{\phi}(h)p_h
$$
for some points $\{p_h,\  h\in H(G)\}\subset C_v$; we set $P_v=\{p_h,\  h\in H(G)\}$.

Indeed, the degree of the ramification divisor of a degree-$m$
 morphism from a curve of genus $w(v)$ to $\PP^1$ of  is equal to $2(m-1+w(v))$. Therefore assumption (\ref{rest2})
guarantees that the ramification conditions we are imposing are compatible; now as $\phi$ is of Hurwitz type, the  Riemann
Existence theorem  yields that such an $\alpha_v$ exists; see Remark~\ref{Riemrk}. Observe that $\alpha_v$ may have other  ramification, in which case we  can easily impose that   any  extra  ramification  and branch  point 
lie   $C_v\smallsetminus P_v$,
 respectively in $D_u\smallsetminus Q_u$, and that they are all simple.

Now that we have the  pointed  curves $(C_v,P_v)$ for every $v\in V(G)$
such that  $C_v$ is a smooth curve of genus $w(v)$ 
we can define $X$:
$$
X:=\frac{\sqcup_{v\in V(G)}C_v}{\{p_h=p_{\ov{h}},\  \forall h\in H(G)\}},
$$
so, $(G,w)$ is the dual graph of $X$.
  
Let us prove that the  morphisms $\{\alpha_v,\  \forall v\in V(G)\}$ glue to a morphism $\alpha:X\to Y$.
It suffices to check that
for    every pair $(p_h,p_{\ov{h}})$    we have $\alpha_v(p_h)=\alpha_{\ov{v}}(p_{\ov h})$,
where   $p_h\in C_v$ and $p_{\ov{h}}\in C_{\ov{v}}$.
  We have $\alpha_v(p_h)=q_{\phi(h)}$ and  $\alpha_{\ov{v}}(p_{\ov{h}})=q_{\phi(\ov{h})}$.
Now, looking at the   involution of $H(T)$  (see subsection~\ref{gcdef}), we have
 $\phi (\ov{h})=\ov{(\phi (h))}$, and hence 
$\alpha:X\to Y$ is well defined. 
  
We now show that $\alpha$ is a covering. It is obvious that   $\alpha^{-1}(Y_{\rm sing})=\sing.$ Next, for every node $N_e$ of $X$, the ramification indices at the two branches, $p_h,p_{\ov{h}}$ where $[h,\ov{h}]=e$, are equal, as they are equal to
 $r_{\phi}(h)$ and $r_{\phi}(\ov{h})$.
As we have imposed that $\alpha_{v}$ has only ordinary ramification
points away from the nodes of $X$, 
  condition (\ref{ac3c})
  of Definition~\ref{acdef} is satisfied. 
Therefore   
  the   map 
  $ 
  \alpha: X\to  Y
  $ is a covering;   obviously $ \alpha$
  has $\phi$ as dual graph-map. 
\end{proof}

To deduce Theorem~\ref{main}  from the previous Proposition we will need to construct a suitable  homomorphism from a given morphism of Hurwitz type,
which is done in the next Lemma.

\begin{lemma}
\label{lemma-hom}
Let $\phi:(G,w)\to T$ be a   degree-$d$ morphism of Hurwitz type.
Then there exists a degree-$d$ homomorphism $\hphi:(\hG,\hw)\to \hT$ of Hurwitz type fitting in a commutative diagram
\begin{equation}\label{diagH}
\xymatrix{
\hG\ar[d]\ar[r]^{\hphi}   & \hT\ar[d] \\
G \ar[r]^{\phi} &T
} 
\end{equation}
whose vertical arrows are edge contractions, and such that  $(\hG,\hw)$ is equivalent to $(G,w)$.
\end{lemma}

\begin{proof}
The picture after the proof illustrates the forthcoming construction.
Since $(G^0,w^0)$  is equivalent to $(G,w)$ we can assume $G$ loopless.
Consider the set of ``vertical" edges of $\phi$:
 $$
\Ev_{\phi}(G):=\{e\in E(G):\  \phi(e)\in V(G')\}
$$
 and set    $\Eh_{\phi}(G):= E(G)\smallsetminus \Ev_{\phi}(G)$. Of course, if $\Ev_{\phi}(G)=\emptyset$ there is nothing to prove.
 So,
let  $e\in \Ev_{\phi}(G)$ 
and $v_1,v_2$ be its endpoints.
We set $u=\phi(v_1)=\phi(v_2)=\phi(e)$ and write
\begin{equation}
\label{verte}
\phi_V^{-1}(u)=\{v_1,v_2, \ldots, v_n\}
\end{equation}
with $n\geq 2$ and the $v_i$ distinct.
Set $m_i:=m_{\phi}(v_i)$ for $i=1,\ldots, n$. 
 
We   begin by constructing  $\hG$.
First,  we   insert a weight zero vertex  $\hv_e$ in the interior of $e$, and
  denote by $\he_1,\he_2$ the two edges adjacent to it.
Next,   we attach $m_1-1$ leaves at $v_1$, $m_2-1$ leaves at $v_2$, and
  $m_i$ leaves at $v_i$ for all $i\geq 3$; all these leaf-vertices are given weight zero.
We denote the $j$-th leaf-edge attached to $v_i$  by $l^{(i)}_{e,j^{(i)}}$
 and its leaf-vertex by $w^{(i)}_{e,j^{(i)}}$, with
 $j^{(i)}=1,\ldots, m_i-1$ if $i=1,2$ and 
$j^{(i)}=1,\ldots, m_i$ if $i\geq 3$.

We repeat this construction for every  $e\in \Ev_{\phi}(G)$, and we denote the so obtained graph by $\hG$.
We have identifications
$$
E(\hG)=   \Eh_{\phi}(G)   \sqcup\{\he_1,\he_2,\  \forall e\in \Ev_{\phi}(G)\}\sqcup  \{\l^{(i)}_{e,j^{(i)}}\  \  \forall e\in \Ev_{\phi}(G), \forall i, \forall j^{(i)}\}
$$
 and
$$
V(\hG)= V(G)  \sqcup\{\hv_e,\  \forall e\in \Ev_{\phi}(G)\}\sqcup \{w^{(i)}_{e,j^{(i)}}\  \  \forall e\in \Ev(_{\phi}G),
\forall i, \forall j^{(i)}\}.
$$
There is a contraction $\hG\to G$ given by contracting, for every $e\in \Ev_{\phi}(G)$, the edge $\he_1$ and all leaf edges $\l^{(i)}_{e,j^{(i)}}$. It is clear that $G$ and $\hG$ are equivalent.

Let us now construct $\hT$;
for every $e\in \Ev(G)$ we add to $T$ a leaf based at $u=\phi(e)$;
we denote by $\widehat{l}_e$,
and $\widehat{w}_e$ the 
  edge and  vertex 
of this leaf. We let $\hT$ be the tree   obtained after repeating this process
for every $e\in \Ev_{\phi}(G)$. 
There is a contraction $\hT\to T$ 
given by contracting  all leaf edges $\widehat{l}_e$.

Let 
 $G':=G-\Ev_{\phi}(G)$, so that $G'$ is also a subgraph of $\hG$.
Denote by $ \phi':G' \to T$ the restriction of $\phi$ to $G'$;
observe that $\phi'$ is a harmonic homomorphism.
To construct $\hphi:\hG\to  \hT$ 
 we extend
 $\phi'$   as follows.
For every $e\in \Ev_{\phi}(G)$ we set, with the above notations,
$$
\hphi(\he_1)=\hphi(\he_2)=\hphi(l^{(i)}_{e,j^{(i)}})=\widehat{l}_e
$$
and
 $$
\hphi(\hv_e)=\hphi(w^{(i)}_{e,j^{(i)}})=\widehat{w}_e 
$$
for every $i$ and $j^{(i)}$.
Finally, we define the indices of $\hphi$ 
\begin{displaymath}
r_{\hphi}(\he)=\left\{ \begin{array}{ll}
r_{\phi}(\he)  &\text{ if }   \he\in \Eh_{\phi}(G)\\
\  1 &\text{ otherwise.}\\
\end{array}\right.
\end{displaymath}
It is clear that $\hphi$  is a homomorphism  and that   diagram \eqref{diagH} is commutative.

Let us check  that $\hphi$   is pseudo-harmonic. 
Pick $e\in \Ev(G)$.
Consider a leaf vertex $w^{(i)}_{e,j^{(i)}}$ of $\hG$.
Then it is clear that condition (\ref{imeq}) holds  with $m_{\hphi}(w^{(i)}_{e,j^{(i)}})=1$.
Next, consider a vertex $\hv_e$. It is again clear that condition (\ref{imeq}) holds with $m_{\hphi}(\hv_e)=2$.
Finally, consider the vertices $v_1,\ldots, v_n$ introduced in (\ref{verte}).
Recall that $\hphi(v_i)=\phi(v_i)=u$ and condition (\ref{imeq}) holds for any edge in $E(T)\subset E(\hT)$ adjacent to $u$ with $m_{\hphi}(v_i)=m_i$. We need to check that the same holds 
for the leaf-edges 
  $\widehat{l}_e\in E(\hT)$.
For $v_1$ and any leaf  $\widehat{l}_e$ adjacent to $\hphi(v_1)$ we have 
$$
\sum_{\he\in E_{v_1}(\hG): \hphi (\he)=\widehat{l}_e}r_{\hphi}(\he)=\sum_{j^{(1)}=1}^{m_1-1} r_{\hphi}(l^{(1)}_{e,j^{(1)}})+r_{\hphi}(\he_1)=m_1-1+1=m_1,
$$
(as  $r_{\hphi}(l^{(1)}_{e,j^{(1)}})=r_{\hphi}(\he_1)=1$)  Similarly for $v_2$. Next,
for $v_i$ with  $i=3,\ldots, n$ we have
$$
\sum_{\he\in E_{v_i}(\hG): \hphi (e)=\widehat{l}_e}r_{\hphi}(\he)=\sum_{j^{(i)}=1}^{m_i} r_{\hphi}(l^{(i)}_{e,j^{(i)}}) =m_i.
$$
Since   $\phi'$   is    pseudo-harmonic
  there is nothing else to check; hence $\hphi$ is pseudo-harmonic.  Now,  to prove that $\hphi$ is harmonic we must check that  condition \eqref{rest2}
holds; 
since $\phi'$ is harmonic,
this follows immediatly  from the fact that
 the index of $\hphi$ at each of the  new edges is 1.
 
Finally, to prove that $\hphi$ is of Hurwitz type,  pick a vertex  of $\hG$; if this vertex
  is of type $\hv_e$ or $w_{e,j^{(i)}}^{(i)}$ then the associated partition set contains only the trivial partition,
  and hence it is obviously contained in some partition set of Hurwitz type.
The remaining case is that of a vertex $v$ of $G$.
Then  either $\mP(\phi,v)= \mP(\hphi,v)$ (if $v$ is not adjacent to $e\in \Ev$),
or $\mP(\hphi,v)$ is obtained by adding the trivial partition to $\mP(\phi,v)$;
in both cases, since by hypothesis $\mP(\phi,v)$ is contained in a partition set of Hurwitz type,
so is $\mP(\hphi,v)$.
\end{proof}
The following picture illustrates  $\hphi$ for a 3-gonal morphism $\phi$. All indices of $\phi$ are set  equal to 1, with the exception of  the vertical edge $e$ for which $r_{\phi}(e)=0$.
\begin{figure}[h]
\begin{equation*}
\xymatrix@=.5pc{
&&&&&&&&&&&&&&&&&*{\circ}\ar@{-}[dr]^{l_e}&&&&&&&&&&&&&&&\\
G=&  
*{\bullet} \ar @{-}@/_.5pc/[rrr]\ar @{-}@/^.5pc/[rrr]^(.9){v_1} &&& *{\bullet}\ar@{-}[dd]_{e}\ar@{-}@/_.5pc/[rrr] _(1){}\ar@{-} @/^.5pc/[rrr]&&&*{\bullet}&&&&&&&\hG=&
*{\bullet} \ar @{-}@/_.5pc/[rrr]\ar @{-}@/^.5pc/[rrr]^{} &&& *{\bullet}\ar@{-}@/_.5pc/[rrr] _(1){}\ar@{-} @/^.5pc/[rrr]&&&*{\bullet}\\
&&&&&&&&&&&&&&&&&*{\circ}\ar@{-}[ur]^(.1){\hv_e}\ar@{-}[dr]
&&&&&&&&&&&&&&&&&&&&&&&&&&&\\
&*{\bullet}\ar@{-}[rrr]_(1){v_2}&&&*{\bullet}\ar@{-}@/_.6pc/[rrruu]&&&&&&&&&&&*{\bullet}\ar@{-}[rrr]&&&*{\bullet}\ar@{-}@/_.6pc/[rrruu]&&&&&&&&&&&&&&\\
\ar[dd]_{\phi}&&&&&&&&&&&&&&\ar[dd]_{\hphi}&&&&&&&&&&&&\\
&&&&&&&&&&&&&&&& &&&&&&&&&&&&&&&\\
&&&&&&&&&&&&&&&&&*{\circ}\ar@{.}[uuuuuu]^(.1){\widehat{w}_e}&&&&&&&&&&&&&&\\
T=&*{\circ}\ar@{.}[uuuuuu]\ar@{-}[rrr]_(1){u}&&&*{\circ}\ar@{.}[uuuuuu]\ar@{-}[rrr]&&&*{\circ}\ar@{.}[uuuuuu]
&&&&&&&\hT=&*{\circ}\ar@{.}[uuuuuu]\ar@{-}[rrr]&&&*{\circ}\ar@{-}[lu]_{\widehat{l}_e}\ar@{-}[rrr]\ar@{.}[uuuuuu]&&&*{\circ}\ar@{.}[uuuuuu]
}
\end{equation*}
 
\end{figure}

\subsection{Proof of Theorem~\ref{main}}
\label{proofmain}
By Corollary~\ref{pfpart2} we need only  prove the first part of the Theorem.
We first  assume that $G$ is free from loops.

  By hypothesis we have  a non-degenerate,
degree-$d$,  harmonic morphism $\phi: G \to T$ of Hurwitz type,
where $T$ is a tree. We let  $\hphi:\hG \to \hT$ be a degree-$d$,  harmonic homomorphism associated to $\phi$ by Lemma~\ref{lemma-hom},
Now  $\hphi:\hG \to \hT$ satisfies all the assumptions of Proposition~\ref{corr}, 
hence there exists a   covering
 $\alpha:\hX\to \hY$
whose dual graph-map is   $\hphi:\hG \to \hT$. 
We denote by
  $y_1,\ldots, y_b\in Y$   the smooth branch points of $\alpha$. 

 Suppose now that $(G,w)$ is stable;
we claim  that $\alpha$
  is  admissible, i.e. 
  that  $(Y;y_1,\ldots, y_b)$ is  stable. We write $Y=\cup_{\hu\in V(\hT)}D_{\hu}$ as usual.
  For every branch point $y_i$ we attach a
  leg  to $\hT$,
  having  endpoint  $\hu\in V(\hT)$ such that
  $y_i\in D_{\hu}$. 
We must prove that the graph $\hT$ with these  $b$ legs   has no vertex of valency less than 3.
Pick a vertex of $\hT$. 
There are two cases, either  it is a 
vertex $u\in V(T)$ or it is a
leaf vertex $\widehat{w}_e$.

In the first case the preimage of $u$ via $\hphi$ is made of vertices of the original graph $G$.
So, pick  $v\in V(G)$ with   $\phi(v)=u$. The map $\alpha_v:C_v\to D_u$ has degree $m_{\phi}(v)$.
 If $m_{\phi}(v)=1$, then, of course, $C_v\cong \PP^1$ and
we have
$ 
\val(u)\geq \val(v),
$  and $\val(v)\geq 3$
as $G$ is stable; hence $\val(u)\geq 3$ as wanted. Notice  that this is the only place
where we use that $(G,w)$ is stable, the rest of the proof works for any  $d$-gonal graph.
 If $m_{\phi}(v)\geq 2$
then the map $\alpha_v$  has at least two branch points, 
each of which corresponds to  a leg  adjacent to $u$. If  $\alpha_v$ has more than two branch points, then $u$ has more than two legs adjacent to it, hence we are done;
  if $\alpha_v$ has exactly two branch points, then, by Riemann-Hurwitz,  $C_v\cong \PP^1$ and hence
 $C_v\subsetneq X$ as $X$ has genus $\geq 2$. Therefore 
  $C_v\cap \ov{X\smallsetminus C_v}\neq \emptyset$, and hence
  there is at least one edge
of $T$ adjacent to $u$, hence $\val(u)  \geq 3$.

Now consider a vertex of type $\widehat{w}_e$. By construction,
its preimage contains the vertex $\hv_e$, for which $m_{\hphi}(\hv_e)=2$;
  hence the corresponding component of $\hX$ maps two-to-one to the component
corresponding to $\widehat{w}_e$, and hence there are at least 2 legs attached to $\widehat{w}_e$ (corresponding to the two branch points).
There is also at least one edge because, as before, $\widehat{w}_e$ is not an isolated vertex of $\hT$.
So, $\val (\widehat{w}_e)\geq 3$. This proves that $\alpha$ is an admissible covering.

Now, $\hX$ is a curve whose dual graph is $(\hG,\hw)$.
Its stabilization   is a stable curve, $X$, whose  dual graph
  is clearly the original $(G,w)$. As we already mentioned, the fact that $X$ is $d$-gonal follows from \cite[Sect 4]{HM}, observing that $X$ is the image of the admissible covering 
$\alpha:\hX\to (Y;y_1,\ldots, y_b)$ under the morphism  (\ref{acmap}). 
This concludes the proof in case $(G,w)$ is stable and loopless.

Now let us drop the stability assumption on $(G,w)$.
If $\alpha$ is admissible, the previous  argument yields that the stabilization of $\hX$ is $d$-gonal. But the stabilization of $\hX$ is the same as the stabilization of $X$, hence we are done. 

Suppose $\alpha$ is not admissible; then there are two cases.
First case: $\hT$ has a vertex $u$ of valency $1$. By the previous part of the proof this can happen only if every vertex $v\in \phi_V^{-1}(u)$ has valency $1$ and   $\alpha$ induces an isomorphism $C_v\cong \PP^1$; such components of $\hX$ are called rational tails.  We now remove
the component $D_u$ from $\hY$, and 
 all the rational tails mapping to $D_u$  from   $\hX$.
 Observe that this  operation does not change the stabilization of $\hX$.
This corresponds to removing one leaf from $\hT$ and all its preimages (all leaves) under $\phi$.
We repeat this process until there are no 1-valent vertices left.  

Second case, $\hT$ has a vertex $u$ of valency $2$. Again by the previous part
 this   happens only if  every $v\in \phi_V^{-1}(u)$ has valency $ 2$ and   $\alpha$ induces an isomorphism $C_v\cong \PP^1$.  We collapse  the component $D_u$ of $\hY$ and   all the exceptional components  of $\hX$ mapping to $D_u$. Again, this  operation does not change the stabilization of $\hX$.
We repeat this process until there are no 2-valent vertices left. 

In this way
we arrive at two curves $X'$ and $(Y';y_1,\ldots, y_b)$,   the latter being stable,
endowed with a covering $\alpha':X'\to Y'$ induced by $\alpha$, by construction; indeed
the process did not touch the   branch points   $y_1,\ldots, y_b$, which are now the smooth branch points of $\alpha'$.
The covering $\alpha'$ is admissible, hence the stabilization of $X'$ is $d$-gonal (as before).
Since the stabilization of $X'$ is equal to the stabilization of $X$ we are done.
The loopless case is now proved.

\

We now   suppose that $G$ has some loop; let $(G^0,w^0)$ be its loopless model.
By Definition~\ref{dgdef}, $(G^0,w^0)$ is  $d$-gonal. The previous part yields that there exists a curve $X^0$ whose dual graph is $(G^0,w^0)$ and whose stabilization is $d$-gonal.
Since the stabilization of $X$ is equal to the stabilization of $X^0$ we are done.
Theorem~\ref{main} is proved.
\qed
 
 \begin{remark}
\label{2ex}
{\it Hyperelliptic and 2-gonal graphs.}
It is easy to construct  hyperelliptic (i.e. divisorially $2$-gonal) graphs that are not $2$-gonal;
for example the weightless graph $G$ in Example~\ref{binex}  for $n\geq 3$.

On the other hand every 2-gonal stable graph is hyperelliptic, by Theorem~\ref{main} and Proposition~\ref{mainconvh};
 see also Theorem~\ref{bridgethm}. More generally, using Remark~\ref{BNext} one can prove directly that if a graph admits a pseudo-harmonic morphism of degree $2$ to a tree, then it is hyperelliptic. We omit the details.
\end{remark}

\

\begin{example}
\label{pdex} {\it A  $3$-gonal graph which is not divisorially $3$-gonal.}

In the following picture we have a pseudo-harmonic morphism $\phi$ of degree $3$ from a weightless graph $G$ of genus $5$. There is one edge, joining $v_2$ and $v_3$, where the index is $2$, and all other edges have index $1$. 
The graph $G$ is easily seen to be  $3$-gonal, but not divisorially $3$-gonal, i.e.
$W^1_3(G)=\emptyset$. We omit the details.
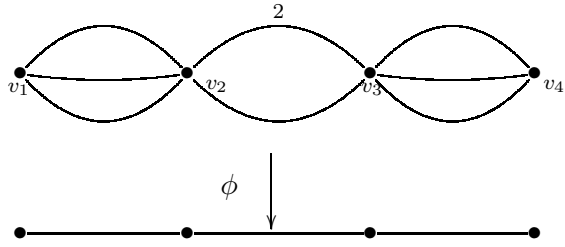
\begin{figure}[h]
\begin{equation*}
\xymatrix@=.5pc{
 &\\
  &&*{\bullet}\ar @{-} @/_.2pc/[rrrrr]_(0.01){v_1} \ar @{-} @/_1.5pc/[rrrrr] _(1){\  v_2}\ar@{-} @/^1.5pc/[rrrrr]
&&&&& *{\bullet}  \ar @{-} @/_1.5pc/[rrrrr] \ar @{-} @/^1.5pc/[rrrrr]^{2} &&&&& *{\bullet}\ar @{-} @/_.2pc/[rrrrr]_(0.03){v_3} \ar @{-} @/_1.5pc/[rrrrr] _(1){v_4}\ar@{-} @/^1.5pc/[rrrrr]
&&&&& *{\bullet} &&&
\\
  &&&&&&&&&\\
&&&&&&&&  &\ar @{->}[dd] &\\
&&&&&&&&  \phi&&\\
  &&*{\bullet}\ar@{-}[rrrrr]&&&&& *{\bullet} \ar@{-}[rrrrr]&&&&& *{\bullet}\ar@{-}[rrrrr]&&&&& *{\bullet}
}
\end{equation*}
\end{figure}
\end{example}

\begin{example}\label{dex} {\it A divisorially $3$-gonal graph  which is not   $3$-gonal.}
 In the graph $G$ below, weightless of genus $5$, we have
 $$
 3v_1\sim 3v_2\sim-v_2+2v_0+2v_3\sim3v_3\sim v_0+v_2+v_3 \sim 3v_4
 $$
 so the graph is divisorially $3$-gonal.

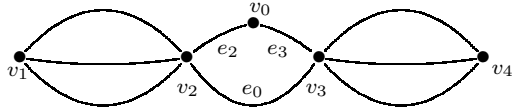
\begin{figure}[h]
\begin{equation*}
\xymatrix@=.5pc{
 &&&&&&&&&*{\bullet}\ar @{-} @/^0.15pc/[rrd]^(0.03){v_0}_(0.5){e_3} \\
  &&*{\bullet}\ar @{-} @/_.2pc/[rrrrr]_(0.01){v_1} \ar @{-} @/_1.5pc/[rrrrr] _(0.9){v_2}\ar@{-} @/^1.5pc/[rrrrr]
&&&&& *{\bullet}  \ar @{-} @/_1.5pc/[rrrr]^{e_0} \ar @{-} @/^0.15pc/[rru]_(0.5){e_2} &&&& *{\bullet}\ar @{-} @/_.2pc/[rrrrr] \ar @{-} @/_1.5pc/[rrrrr] _(1){v_4}_(0.1){v_3}\ar@{-} @/^1.5pc/[rrrrr]
&&&&& *{\bullet} &&&
\\
}
\end{equation*}
\end{figure}

Let us show that $G$ does not admit a non-degenerate pseudo-harmonic morphism of degree $3$ to a tree. 
By contradiction, let $\phi:G\to T$ be such a morphism. Then the edges adjacent to $v_1$ cannot get contracted (if one  of them is contracted, all of them will be contracted, for $T$ has no loops; but if all of them get contracted then $m_{\phi}(v_1)=0$, which is not possible).
Therefore the three edges adjacent to $v_1$ are all mapped to the unique edge, $e'_1$, joining $\phi(v_1)$ with $\phi(v_2)$.
Similarly, the edges adjacent to $v_4$ are all mapped to the unique edge $e'_2$
joining $\phi(v_4)$ with $\phi(v_3)$. Therefore, as $\phi$ as degree $3$, all edges between $v_1$ and $v_2$,
and all edges between $v_3$ and $v_4$ have index 1, hence $m_{\phi}(v_1)=m_{\phi}(v_2)=m_{\phi}(v_3)=m_{\phi}(v_4)=3$. 

Now, if $\phi(v_2)=\phi(v_3)$ then one easily checks that $e_0$ is contracted and $e_2$, $e_3$ are  mapped to the same edge $e'_3$ of $T$, which is different from $e'_1$ and $e'_2$. Therefore  we
have $1\leq r_{\phi}(e_i)\leq 2$ for $i=1,2$.
But then by \eqref{imeq} we have
$$
m_{\phi}(v_2)=\sum_{e\in E_{v_2}(G): \phi(e)=e'_3}r_{\phi}(e)=r_{\phi}(e_2)\leq 2
$$
and this is a contradiction.

It remains to consider the case $\phi(v_2)\neq\phi(v_3)$,
let $e'_0=\phi(e_0)$. Then   $v_0$ is either mapped to $\phi(v_2)$ by contracting $e_2$, or to  $\phi(v_3)$  by contracting $e_3$�
(for otherwise $T$ would not be a tree). With no loss of generality, set
$\phi(v_2)=\phi(v_0)$ so that $r_{\phi}(e_2)=0$. Now, since $\phi(e_3)=\phi(e_0)=e'_0$  we have $r_{\phi}(e_0)\leq 2$. Hence
$$
m_{\phi}(v_2)=\sum_{e\in E_{v_2}(G): \phi(e)=e'_0}r_{\phi}(e)=r_{\phi}(e_0)\leq 2
$$
and this is a contradiction.
\end{example}
 \section{Higher gonality and applications to tropical curves}
 
\subsection{Basics on tropical curves}
\label{tropsec}
A (weighted) tropical curve is a weighted metric graph $\Gamma=(G,w,\ell)$
where $(G,w)$ is a weighted graph and $\ell:E(G)\to \R_{>0}$.
The divisor group $\Div(\Gamma)$ is, as usual, the free abelian group generated by the points of $\Gamma$ (viewed as a metric space).  
The weightless case has been carefully studied in \cite{GK}, for example;   the general case 
has been recently treated in
 \cite{AC}, to which we refer for the definition  of the  rank $r_{\Gamma}(D)$
of any $D\in \Div(\Gamma)$ and its basic properties. Here we just need the following facts.
 Given $\Gamma=(G,w,\ell)$ we introduce the tropical curve $\Gamma^w=(G^w,\mo,\ell^w)$
such that $G^w$ is as in Definition~\ref{varG}, the weight function is zero (hence denoted by $\mo$), and $\ell^w$ is the extension of $\ell$ such that $\ell^w(e)=1$ for every $e\in E(G^w)\smallsetminus E(G)$.
We have a natural commutative diagram 
\begin{equation}\label{diag2}
\xymatrix{
\Div(G,w)\ar@{^{(}->}[d]\ar@{^{(}->}[r]   & \Div(G^w)\ar@{^{(}->}[d] \\
\Div(\Gamma) \ar@{^{(}->}[r] & \Div(\Gamma^w)
} 
\end{equation}
the above injections will be viewed as inclusions in the sequel.
Then, for any $D\in \Div (\Gamma)$  we have, by \cite[Sect. 5]{AC}
 \begin{equation}
\label{ACRR}
 r_{\Gamma}(D)=r_{\Gamma^w}(D).
 \end{equation} So, the horizontal arrows of the above diagram preserve the rank. If the length functions on $\Gamma$ and $\Gamma^w$ are identically equal to $1$,   then, by 
 \cite[Thm 1.3]{luo}, also the vertical arrows of the diagram preserve the rank.

For a tropical curve $\Gamma$ we denote by $W^r_d(\Gamma)$ the set of equivalence classes of divisors of degree $d$ and rank at least $r$; we say that $\Gamma$ is $(d,r)$-gonal if
$W^r_d(\Gamma)\neq \emptyset$.

The moduli space of equivalence classes of tropical curves of genus $g$ is denoted by
$\Mgt$, and the locus in it of curves whose underlying weighted graph is $(G,w)$ is denoted by $\Mt(G,w)$. This gives a partition $$\Mgt=\sqcup \Mt(G,w)$$ indexed by all stable graphs $(G,w)$ of genus $g$.
\subsection{From   algebraic gonality to combinatorial   and tropical gonality}
\begin{thm}
\label{mainconv}
Let $X\in \ov{M^r_{g,d}}$
and let $(G,w)$ be the dual graph of $X$. Then
\begin{enumerate}[{(\bf A)}]
\item
there exists a refinement $(\hG,\widehat{w})$ 
of  $(G,w)$, such that $W^r_d(\hG,\widehat{w})\neq \emptyset$;
\item
there exists a tropical curve $\Gamma\in \Mt(G,w)$
such that $W^r_d(\Gamma)\neq \emptyset$.
\end{enumerate}
\end{thm}
\begin{proof}
By hypothesis there exists a family of curves, $f:\X\to B$, with  $B$    smooth, connected, of dimension one,
 such that there is a point $b_0\in B$ over which the fiber of $f$  is isomorphic to $X$, and the fiber over any other point of $B$ is a smooth curve
 whose $W^r_d$ is not empty. 
 In the sequel we will work up to replacing $B$ by an open neighborhood of $b_0$, or by  an \'etale covering.
Therefore we will also  assume that $f$ has a section.

For every $b\in B^* = B\smallsetminus \{b_0\}$ we have $W^r_d(X_b)\neq 0$ 
($X_b$ is the fiber of $f$ over $b$).
Write $f^*:\X^*  \to B^*$
for the smooth family obtained by restricting $f$ to $\X\smallsetminus X_0$. Recall that
as $b$ varies in $B^*$ the   $W^r_d(X_b)$
form a family (\cite[Sect. 2]{ACo} or \cite[Ch. 21]{gac}),  i.e. there exists a morphism of schemes
\begin{equation}
\label{Wrdfam}
W^r_{d, f^*}\to B^*
\end{equation} 
whose fiber over $b$ is $W^r_d(X_b)$.

 Up to replacing $B$ by a finite covering possibly ramified only over $b_0$,
 we may assume that the base change of the morphism \eqref{Wrdfam} has a section.
  The base change of $f$ to this covering may be singular (or even non normal) over $b_0$, but will still have smooth fiber away from $b_0$.
Let $h:\ZZ\to B$ be the desingularization of the normalization of this base change of $f$.
Then the fiber of $h$ over $b_0$
is a semistable curve $Z_0$ whose stabilization is $X$; all remaining fibers are isomorphic
to the original fibers of $f$.
By construction, the morphism
\begin{equation}
\label{Wrdfamh}
W^r_{d,h^*}\to B^*
\end{equation} 
has a section, $\sigma$. By our initial  assumption $h:\ZZ\to B$ is   endowed with a section, hence, by \cite[Prop. 8.4]{BLR},
$\sigma$  corresponds to a
  line bundle $\L^*\in \Pic \ZZ^*$.
Since  $\ZZ$ is nonsingular  $\L^*$ extends to some line bundle $\L$ on $\ZZ$, and we have,  for every $b\in B$:
  $$r(Z_b,\L_{|Z_b})\geq r.$$

 Let $(\hG,\widehat{w})$ be the dual graph of $Z_0$.
We can apply
the weighted specialization Lemma   \cite[Thm 4.9]{AC}
  to $\ZZ\to B$ with respect to the line bundle $\L$. 
  This gives, viewing the multidegree $\mdeg \ \L_{|Z_0}$ as a divisor on $\hG$,
  $$
  r_{(\hG,\widehat{w})}(\mdeg \ \L_{|Z_0})\geq r(Z_b, \L_{|Z_b})\geq r
  $$
  and therefore
 $W^r_d(\hG,\widehat{w})\neq \emptyset.$ 

Now, by construction
$(\hG,\widehat{w})$ a refinement of $(G, w)$ (the dual graph of $X$). Hence the first part is proved.

For the next part, consider the tropical curve $\widehat{\Gamma}=(\hG,\widehat{w},\widehat{\ell})$
 with $\widehat{\ell}(e)=1$ for every $e\in E(\hG)$. 
 Let $D\in W^r_d(\hG,\widehat{w})$. Then  $D$ is also a divisor on $\widehat{\Gamma}$ (cf. Diagram \eqref{diag2}).
 We claim that $r_{\widehat{\Gamma}}(D)\geq r$.
 
 We have, by definition, 
 $$
 r\leq r_{(\hG,\widehat{w})}(D)=r_{\hG^{\hw}}(D).
 $$
 Let $\widehat{\Gamma}^{\hw}=(\hG^{\hw},\mo, \widehat{\ell}^{\hw})$
 be the tropical curve such that $\widehat{\ell}^{\hw}(e)=1$ for every $e\in E(\Gamma^{\hw})$; so $D$ is also a divisor on $\widehat{\Gamma}^{\hw}$.
By  
 \cite[Thm 1.3]{luo}, we have
 $$
 r_{\hG^{\hw}}(D)=r_{\widehat{\Gamma}^{\hw}}(D).
 $$
 Now, as we noticed in \eqref{ACRR} we have
$$
 r_{\widehat{\Gamma}^{\hw}}(D)=r_{\widehat{\Gamma}}(D). 
 $$
The claim is proved; therefore
  $ 
  W^r_d(\widehat{\Gamma})\neq \emptyset.
  $ 

The supporting graph $(\hG,\widehat{w})$ of $\widehat{\Gamma}$ is not necessarily stable;  its stabilization,  obtained    by removing
every 2-valent vertex of weight zero, is the original $(G, w)$, so that
 $\widehat{\Gamma}$ is tropically equivalent to a curve 
$\Gamma\in \Mgt(G,w)$.
Since  the underlying metric spaces of $\Gamma$ and $\widehat{\Gamma}$  coincide, we have
 $$
 W^r_d({\Gamma})=  W^r_d(\widehat{\Gamma})\neq \emptyset.
  $$
The statement is proved.
\end{proof}
 \begin{cor}
 \label{refcor}
Every   $d$-gonal stable weighted graph 
admits a divisorially $d$-gonal refinement.

\end{cor}
\begin{proof}
Let $(G,w)$ be a   $d$-gonal stable graph. By Theorem~\ref{main} there exists $X\in \ov{M^1_{g,d}}$ whose dual graph is $(G,w)$. By Theorem~\ref{mainconv} we are done.
\end{proof}

The proof of Theorem~\ref{mainconv} gives a more precise result, to state which we need some further terminology. 

Let $X$ be any curve. A {\it one-parameter smoothing}  of $X$ is a 
morphism  $f:\X\to (B,b_0)$, where $B$ is    smooth connected with $\dim B=1$, 
$b_0$ is a point of $B$ 
 such that $f^{-1}(b_0)=X$,  and all other fibers of $f$  are smooth curves.
By definition, $\X$ is a surface having only  singularities of type $A_n$ at the nodes of $X$. To  $f$
we associate the following length function $\ell_f$ on $G_X$:
$$
\ell_f:E(G_X)\la \R_{>0};\quad \quad e\mapsto n(e)
$$
where $n(e)$ is the integer defined by the fact that $\X$ has a singularity  of type $A_{n(e)-1}$  at the node of $X$ corresponding to $e$. In particular, if $\X$ is nonsingular, then $\ell_f$ is constant equal to one.
This defines the following tropical curve associated to $f$:
$$
\Gamma_f=(G_X,w_X, \ell_f).
$$
Similarly, we   define a refinement of the dual graph of $X$
by inserting $n(e)-1$ vertices of weight zero in  $e$, for every $e\in E(G_X)$;
we denote this refinement by $(G_f,w_f)$.  Now, 
if $\ZZ\to \X$ is the minimal resolution of singularities   and $h:\ZZ\to B$ the composition with $f$, then  
$(G_f,w_f)$ is the dual graph
of the fiber of $h$ over $b_0$;  
we denote by $X_f$ this fiber. 

For example, the surface $\X$ is nonsingular if and only if $X=X_f$, if and only if $(G_X,w_X)= (G_f,w_f)$

The following is a consequence  the proof of Theorem~\ref{mainconv},
where $X_f$ corresponds to the curve
  $Z_0$, while   $(G_f,w_f)=(\hG,\hw)$, and $\Gamma_f=\Gamma.$

\begin{prop}
\label{mainconvgen}
Let $f:\X\to (B,b_0)$ be a one-parameter smoothing of the curve $X$.
If the general fiber of $f$ is $(d,r)$-gonal
(i.e. if $W^r_d(f^{-1}(b))\neq \emptyset$ for every $b\neq b_0$) then the following facts hold.
\begin{enumerate}
\item
$W^r_d(G_f,w_f)\neq \emptyset$.
\item
$W^r_d(\Gamma_f)\neq \emptyset$.
\item
$W^r_d(X_f)\neq \emptyset$.
\end{enumerate}
\end{prop}
    \begin{remark}
 \label{skeleton}
The tropical curve $\Gamma_f$  may be interpreted as a Berkovich skeleton of the generic fiber $\X_K$ of $\X\to B$, where $K$ is the function field of $B$ (note that $\Gamma_f$ depends on $\X$). Then the   theorem says that the Berkovich skeleton of a $(d,r)$-gonal smooth algebraic curve over $K$ is
a $(d,r)$-gonal tropical curve.
\end{remark}

 \section{The hyperelliptic case}
 \label{hypsec}
 \subsection{Hyperelliptic weighted graphs}
 Recall that a graph is hyperelliptic if it has a divisor of degree two and rank one.
 Hyperelliptic graphs free from loops and weights have been thoroughly studied in \cite{BN}. In this subsection we    extend some of  their results to weighted graphs admitting loops.

Recall the notation of Definition~\ref{varG}.
We will use the following  terminology.
A   2-valent vertex of     is said to be   {\it special} if  its removal  creates a loop.  For example, given  $(G,w)$, every vertex in $V(G^w)\smallsetminus V(G)$ is special.

  \begin{lemma}
\label{lmhyp} Let 
 $(G,w)$ be a weighted graph of genus $g$. Then
 $(G,w)$ is hyperelliptic if and only if so is $G^w$ if and only if so is $(G^0,w^0)$.
\end{lemma}
\begin{proof} By Remark~\ref{g01} we can assume $g\geq 2$.
By definition, if $G$ is hyperelliptic so is $G^w$. Conversely, assume $G^w$ hyperelliptic and let 
$D\in \Div(G^w)$ be an effective divisor of degree $2$ and rank $1$. If $\supp D\subset V(G)$ we are done, as $r_{(G,w)}(D)=r_{G^w}(D)$. Otherwise, suppose $D=u+u'$ with
$u\in V(G^w)\smallsetminus V(G)$. So, $u$ is a special vertex whose   removal    creates  a loop based at a vertex $v$ of $G$.
As  $r_{G^w}(u+u')=1$, it is clear that $u'\neq v$ (e.g. by  \cite[Lm. 2.5(4)]{AC}), and
a trivial direct checking 
yields that    $u'=u$.
Moreover, we have  $2u\sim 2v$ and hence
$r_{G^w}(2v)=1$, by   \cite[Lm. 2.5(3)]{AC}).

As  $(G^0)^{w^0}=G^w$, the second double implication follows the first.
\end{proof}
Let $e$ be a non-loop edge of a weighted graph $(G,w)$
and let $v_1,v_2\in V(G)$ be its endpoints. Recall that
the {\it (weighted) contraction} of $e$ is defined as the graph $(G_e,w_e)$
such that $e$ is contracted to a vertex $\oov$ of $G_e$,  and 
$w_e(\oov)=w(v_1)+w(v_2)$, whereas $w_e$ is equal to $w$ on every remaining vertex of $G_e$.

We denote by $(\ooG,\oow)$ the  2-edge-connected weighted graph obtained by contracting every bridge of $G$
as described above. 

By \cite[Cor 5.11]{BN} a weightless, loopless graph is hyperelliptic if and only if so is $\ooG$. The following Lemma extends this fact to the weighted case.
\begin{lemma}
\label{hypb}
Let  $(G,w)$ be a loopless   weighted graph of genus at least 2.
Then $(G,w)$ is hyperelliptic if and only if so is $(\ov{G},\ov{w})$.
\end{lemma}

 \begin{proof}
By Lemma~\ref{lmhyp}, $(G,w)$ is hyperelliptic if and only if so is $G^w$. Similarly, $(\ov{G}, \oow)$ is hyperelliptic if and only if so is $\ov{G}^{\oow}$. Now,  $\ov{G}^{\oow}$ is obtained from $G^w$ by contracting all of its bridges (indeed, the bridges of $G$ and $G^w$ are in natural bijection). Therefore, as we said above, 
$G^w$ is hyperelliptic if and only if so is $\ov{G}^{\oow}$. So we are done.
  \end{proof}

Recall, from \cite{BN},   that a loopless, 2-edge-connected, weightless graph $G$ is hyperelliptic if and only if it has an involution $\iota$ such that $G/\iota$ is a tree.
If $G$ has genus at least 2, this involution is unique and will be called the {\it hyperelliptic} involution.
Furthermore, the quotient map $G\to G/\iota$ is a non-degenerate harmonic morphism,
unless $|V(G)|=2$; see \cite[Thm 5.12 and Cor 5.15]{BN}
We are going to generalize this to the weighted case.

\begin{remark}
\label{spvertex}
Let $G$ be a loopless, 2-edge-connected hyperelliptic  graph of genus $\geq 2$
and $\iota$ its hyperelliptic involution.
Let $v\in V(G)$ be a   special vertex whose removal creates a loop  based at the vertex $u$.
Then $\iota(v)=v$, \  $\iota(u)=u$ and $\iota$ swaps the two edges adjacent to $v$.

Indeed, $G/\iota$ is a tree, hence  the two edges adjacent to $v$   are   mapped to  the same edge by 
$G\to G/\iota$.
As $v$ has valency 2 and $u$ has
 valency at least 3 ($G$ has genus at least 2), $\iota$ cannot swap $v$ and $u$.   Hence $\iota(v)=v$ and $\iota(u)=u$. 
\end{remark}

\begin{lemma}
\label{hypw}
Let  $(G,w)$ be a loopless, 2-edge-connected weighted graph of genus at least 2.
Then $(G,w)$ is hyperelliptic if and only if $G$ has an involution $\iota$,
 the {\emph {hyperelliptic involution}},
 fixing every vertex of positive weight and  such that $G/\iota$ is a tree.

 $\iota$ is unique  
 and, if $|V(G)|\geq 3$,  then the quotient $G\to G/\iota$ is a non-degenerate harmonic morphism of degree 2.
\end{lemma} 
 \begin{proof}
 Assume that $G$ has an involution as in the statement; then we extend $\iota$ to an involution $\iota^w$ of $G^w$ by requiring that $\iota^w$ fix  all the (special) vertices in $V(G^w)\smallsetminus V(G)$
 and swap  the two edges adjacent to them.
 It is clear that $G^w/\iota^w$ is the  tree obtained by adding $w(v)$ leaves to the vertex
 of $G/\iota$ corresponding to every vertex $v\in V(G)$. Hence $G^w$ is hyperelliptic, and hence so is $(G,w)$ by Lemma~\ref{lmhyp}.
 
 Conversely, suppose $G^w$ hyperelliptic and let $\iota^w$ be its hyperelliptic involution.
 Let $v\in V(G)\subset V(G^w)$ have positive weight. Then there is a 2-cycle in $G^w$ attached at $v$; let $e^+$ and $e^-$ be its two edges, and $u$ its special vertex. By Remark~\ref{spvertex} we know that
 $\iota^w$ fixes $v$ and $u$ and swaps $e^+$ and $e^-$.  Notice that the image in $G^w/\iota^w$ of  every such 2-cycle is a leaf.
 
We obtain that the restriction of $\iota^w$ to $G$ is an involution of $G$, written $\iota$, fixing all vertices of positive weight. Finally,
the quotient $G/\iota$ is the tree obtained from $ G^w/\iota^w$ by removing all the above leaves, so we are done.

As $G$ is 2-edge-connected, by  Remark~\ref{BNrk} we can apply   some results from \cite{BN}. 
In particular, the uniqueness of $\iota$ follows from  Corollary  5.14. Next, if $|V(G)|\geq 3$ then $G\to G/\iota$
is harmonic and non-degenerate by   Theorem  5.14 and Lemma 5.6.\end{proof}
 
\begin{cor}
\label{binary} Let $(G,w)$ be a loopless, 2-edge-connected  graph of genus at least 2, having exactly two vertices,  $v_1$ and $v_2$.
Then $(G,w)$ is hyperelliptic if and ony if either $|E(G)|=2$, or $|E(G)|\geq 3$ and $w(v_1)=w(v_2)=0$.
\end{cor}
\begin{proof}
Assume  $(G,w)$ hyperelliptic.
Let $|E(G)|\geq 3$;  by contradiction,
suppose $w(v_1)\geq 1$. By Lemma~\ref{hypw} the hyperelliptic involution fixes $v_1$, and hence it fixes also $v_2$; therefore $G/\iota$ has two vertices. Since   there are at least three edges between $v_1$ and $v_2$, such edges fall into at least two orbits under $\iota$, and each such orbit is an edge of  the quotient $G/\iota$, which therefore cannot be a tree. This is a contradiction.
The other implication  is trivial; see  Example~\ref{binex}.
\end{proof}
\subsection{Relating hyperelliptic curves and graphs}
\begin{prop}
\label{mainconvh}
Let $X$ be a hyperelliptic stable curve. Then  its  dual graph  $(G_X,w_X)$ is hyperelliptic.
\end{prop}
\begin{proof}
We write $(G,w)=(G_X,w_X)$ for simplicity.
By Theorem~\ref{mainconv}, there exists a hyperelliptic refinement, $(\hG,\hw)$,  of  $(G,w)$. Then the weightless graph $\hG^{\hw}$ is hyperelliptic. By Lemma~\ref{lmhyp}
it is enough to prove that the weightless graph $G^w$ is hyperelliptic.
Now, one easily checks that $G^w$ is obtained from $\hG^{\hw}$  by removing every  
non-special 2-valent vertex of weight zero, 
  and possibly some special   vertex of weight zero.
On the other  hand, by Lemma~\ref{lmhyp}, the removal of  any special vertex of weight zero does  not
alter being hyperelliptic.
 Therefore 
 $G^w$ is hyperelliptic if so is the graph obtained by removing every 2-valent vertex
 of weight zero
from $\hG^{\hw}$.
This  follows from the following Lemma~\ref{hyp2}.
\end{proof}

\begin{lemma}
\label{hyp2}
 Let $(\hG,\hw)$ be hyperelliptic of genus at least 2 and let $(G,w)$ be the graph obtained
from  $\hG$ by removing every  2-valent vertex of weight zero. Then $G$
  is hyperelliptic.
\end{lemma}
\begin{proof}
By Lemma~\ref{hypb}, contracting bridges does not alter   being hyperelliptic, hence we may assume that $\hG$ is 2-edge-connected.
By Lemma~\ref{lmhyp} up to inserting some special vertices  of weight zero we can also assume that $\hG$ has no loops. 
Finally, we can assume that $\hG$ has at least three vertices, for otherwise the result is trivial.

It suffices to prove that the loopless model $(G^0,w^0)$ (see Definition~\ref{dgdef}) of $(G,w)$ admits an involution $\iota$ fixing every vertex of positive weight and such that $G^0/\iota$ is a tree, 
 by Lemma~\ref{hypw}. As $(\hG,\hw)$ is hyperelliptic, it admits such an involution, denoted by $\hiota$. Recall that the quotient map $\hG \to \hG/\hiota\  $ is a non-degenerate harmonic morphism.

Observe that $G^0$ is obtained from $\hG$ by removing  all the non-special 2-valent vertices
of weight zero.
Let   $\hv\in V(\hG)$ be   such a  vertex and
write  $\he_1, \he_2$  for the edges of $\hG$ adjacent to $\hv$. 
To prove our 
result it suffices to show   that if one removes
from a hyperelliptic   graph either 
 a  non-special 2-valent vertex of weight zero
fixed by the hyperelliptic involution, or a pair of non-special  2-valent vertices swapped by  the hyperelliptic involution, then the resulting graph is hyperelliptic.

First, let   $\hiota (\hv)=\hv$ and let $(G',w')$ be the graph obtained by removing    $\hv$.
We have   $\hiota(\he_1)= \he_2$ (as $\hG \to \hG/\hiota\  $ is non-degenerate),
and   $\hv$ is mapped to a leaf of $\hG/\hiota$. 
Now,  $V(G')=V(\hG)\smallsetminus \{\hv\}$, and 
$E(G')=\{e\} \cup  E(\hG)\smallsetminus \{\he_1, \he_2\}$ where   $e$ is the edge created by removing $\hv$.
We define the involution $\iota'$ of $G'$
by restricting   $\hiota$ on $V(G')$ and on $E(\hG)\smallsetminus \{\he_1, \he_2\}$,
and by  setting  $\iota'(e)=e$. Since $\iota'$ swaps the two 
endpoints of $e$ (because so does $\hiota$), we have that $e$ is contracted to a point by the quotient $G'\to G'/\iota'$. Therefore
 $G'/\iota'$ is the tree
obtained from $\hG/\hiota$ by removing the leaf corresponding to $\hv$. 
It is clear that $\iota'$ fixes all vertices of positive weight, hence  $(G',w')$ is hyperelliptic.

Next, let $\hiota (\hv)=\hv'\neq \hv$;
with  $\hv$ and $\hv'$   non-special   and  2-valent,
then the vertex  of 
$\hG/\hiota$ corresponding to $\{\hv, \hv'\}$ is  2-valent as well. Moreover, $\hv$ and $\hv'$ have weight zero, by Lemma~\ref{hypw}.
Let us show that the graph $(G'',w'')$  obtained by removing  $\hv$ and $\hv'$ is hyperelliptic.
Now $\hiota$ maps $\he_1, \he_2$   to the two edges  adjacent to $\hv'$.
We denote by $e$ and $e'$ the new edges of $G''$. 
 We define $\iota''$ on $V(G'')=V(\hG)\smallsetminus \{\hv, \hv'\}$ by restricting $\hiota$;
next, 
we define $\iota''$ on $E(G'')$ so that $\iota''(e)=e'$ and $\iota''$ coincides with $\hiota$ on the remaining edges. 
It is clear that $\iota''$ is an involution fixing positive weight vertices and such that the quotient $G''/\iota''$ is the tree obtained from
$\hG/\hiota$ by removing the 2-valent vertex corresponding to   $\{\hv, \hv'\}$.
We have thus proved that  $(G'',w'')$ is hyperelliptic.
 The proof is now complete.
 \end{proof}
\begin{thm}
\label{bridgethm}
Let $(G,w)$ be a   stable graph of genus $g\geq 2$. Then 
the following are equivalent.
\begin{enumerate}[{\bf (A)}]
\item
\label{bralg}
$\Ma(G,w)$ contains a hyperelliptic curve.
\item
\label{brbr}
 $(G,w)$ is hyperelliptic and for every $v\in V(G)$ 
the number of bridges of $G$ adjacent to $v$ is at most $2w(v)+2$.
\item
\label{br2}
Assume $|V(G)|\neq 2$; the graph $(G,w)$ is $2$-gonal.
\end{enumerate}
\end{thm}
 \begin{proof}
 \eqref{br2} $\Rightarrow $ \eqref{bralg} by Theorem~\ref{main} and Example~\ref{dleq3}.

\eqref{bralg} $\Rightarrow $ \eqref{brbr}.
Let $X$ be a hyperelliptic curve  such that $(G_X,w_X)=(G,w)$.
Then, by Proposition~\ref{mainconvh}, $(G,w)$ is hyperelliptic.
Let $\alpha:\hX\to Y$ be an admissible covering corresponding to $X$; by Remark~\ref{acrk} \eqref{acdeg2},  $\hX$ is   semistable. 
Therefore the dual graph
of $\hX$, written $(\hG,\hw)$,    is a refinement of $(G,w)$ (as $X$ is the stabilization of $\hX$).

Let $v\in V(G)\subset V(\hG)$ and
  $C_v\subset \hX$ be the component corresponding to $v$, recall that $C_v$ is nonsingular
(by Remark~\ref{acrk}) of genus $w(v)$. Now let  $\he\in E(\hG)$
be a bridge of $\hG$
 adjacent to $v$. Then the corresponding node $N_{\he}$ of $
\hX$
is a
separating node of $\hX$, and hence $\alpha^{-1}(\alpha(N_{\he}))=N_{\he}$.
This implies that the restriction of $\alpha$ to $C_v$ ramifies at the point corresponding to $N_{\he}$.
By the Riemann-Hurwitz formula, the number of ramification points of $\alpha_{|C_v}$ is at most $2w(v)+2$, therefore
 the number of bridges of $\hG$ adjacent to $v$ is at most  $2w(v)+2$.

Now,  by construction,
we have a natural identification 
$ 
E_v(\hG)=E_v(G)
$ 
which identifies  bridges with bridges. Hence also the number of bridges of $G$ adjacent to $v$ is at most  $2w(v)+2$, and we are done.

\

\eqref{brbr} $\Rightarrow $ \eqref{br2} {\it assuming}  $|V(G)|\neq 2$. 
We can assume $|V(G)|\geq 3$ for the case $|V(G)|=1$ is clear; see Example~\ref{acex}.
Let us first assume that $G$ has no loops. By Lemma~\ref{hypb}, the  2-edge-connected graph $(\ooG, \oow)$ is hyperelliptic. 

Suppose $|V(\ooG)|>2$.
By Lemma~\ref{hypw}, $\ooG$ has an involution $\ooiota$ such that 
$$
\oophi:\ooG\la \ooT:=\ooG/\ooiota
$$
 is a non-degenerate harmonic morphism of degree 2, with $\ooT$ a tree.
Let us show that $\oophi$ corresponds to a non-degenerate pseudo-harmonic morphism of degree 2, $\phi:G\to T$,
with $T$  a tree, such that $r_{\phi}(e)=2$ for every bridge $e$. 
Suppose that $G$ has a unique bridge $e$, which is contracted to the vertex $\oov$ of $\ooG$;
let $\ov{u}=\oophi(\oov)\in V(\ooT)$.
Let $T$ be the tree obtained from   $\ooT$ by replacing the vertex $\ov{u}$ by a bridge  $e'$ 
and its two endpoints in such a way that there exists
  a morphism $\phi:G\to T$ mapping $e$ to $e'$  fitting in a
  commutative diagram
\begin{equation}\label{diag1}
\xymatrix{
G\ar[d]_{\phi}  \ar[r]   & \ooG\ar[d]_{\oophi} \\
T \ar[r] & \ooT
}
\end{equation}
where the horizontal arrows are the maps contracting $e$ and $e'$ (it is trivial to check that such a $\phi$ exists).
To make $\phi$ into an indexed morphism of degree 2 we set $r_{\phi}(e)=2$
and we set all other indices to be equal to 1. Since $\oophi$ was harmonic and non-degenerate, we have that
$\phi$ is pseudo-harmonic and non-degenerate.

If $G$ has any number of bridges, we iterate this construction   one bridge at the time. This clearly yields a pseudo-harmonic, degree 2,  non-degenerate morphism $\phi:G\to T$ where $T$ is a tree. 

We claim that condition \eqref{rest2} holds. Indeed,
we have $r_{\phi}(e)=2$ if and only if $e$ is a bridge.
Therefore \eqref{rest2} needs only be verified at the vertices of $G$ that are adjacent to some bridge; notice that for any such vertex $v$ we have $m_{\phi}(v)=2$.
Writing ${\rm{brdg}}(v)$ for the number of bridges adjacent to $v$, we have, as by hypothesis, ${\rm{brdg}}(v)\leq 2w(v)+2$,
$$
\sum_{e\in E_v(G)\cap \Eh _{\phi}(G)}(r_{\phi}(e)-1) 
\leq {\rm{brdg}}(v)\leq 2w(v)+2=
 2(w(v)+m_{\phi}(v)-1). 
$$
This proves that  \eqref{rest2} holds, that is, $(G,w)$ is a  2-gonal graph.
So we are done.

Suppose $|V(\ooG)|=2$, 
hence the bridges of $G$ are leaf-edges.
By Corollary~\ref{binary}, if $|E(\ooG)|\geq 3$, then all the weights are zero, hence, as $G$ is stable, $G=\ooG$,
which is excluded. If $|E(\ooG)|=2$, then the vertices 
  must be fixed by the hyperelliptic involution (for otherwise they would have weight zero by Lemma~\ref{hypw}, contradicting that the genus be at least $2$). 
But then
$\ooG$ has clearly an involution $\ooiota$ swapping its two edges and fixing the two vertices, whose quotient  is a non-degenerate harmonic morphism of degree 2 to a tree, as in the previous part of the proof, which therefore applies also in the present case.

Suppose $|V(\ooG)|=1$. Then $G$ is a tree, hence the identity map $G\to G$ with all indices equal to $2$ is a pseudo-harmonic morphism, $\phi$, of degree $2$. Arguing as in the previous part we get $\phi$ is harmonic; so we are done.

Finally, suppose $G$ admits some loops. 
Let $(G^0,w^0)$ be the loopless model;
then $|V(G^0)|\geq 3$. By the previous part we have that 
  $(G^0,w^0)$ is $2$-gonal, hence so is $(G,w)$.
  
   \

\eqref{brbr} $\Rightarrow $ \eqref{bralg} {\it assuming}  $|V(G)|=2$. 
If $G$ has loops, then  $|V(G^0)|\geq 3$ and we can use   the previous implications
\eqref{brbr} $\Rightarrow $ \eqref{br2}   $\Rightarrow$ \eqref{bralg}. So we   assume $G$ loopless.
By \cite{HM}, 
 hyperelliptic curves with two components are easy to describe.
Let $X=C_1\cup C_2$ with $C_i$ smooth, hyperelliptic  of genus $w(v_i)$ and such  that  $X\in \Ma(G,w)$.
If  $|E(G)|= 1$  for  $X$ to be hyperelliptic it suffices to
 glue $p_1\in C_1$ to $p_2\in C_2$ with $p_i$   Weierstrass point of $C_i$ for $i=1,2$.   

If  $|E(G)|= 2$ for  $X$ to be hyperelliptic it suffices to
  glue $p_1,q_1\in C_1$ to $p_2,q_2\in C_2$ with $h^0(C_i,p_i+q_i)\geq 2$   for $i=1,2$.  

If  $|E(G)|\geq 3$,
by Corollary~\ref{binary}   all  weights are zero.
For  $X$ to be hyperelliptic it suffices to pick two copies of the same rational curve with $|E(G)|$ marked points, and glue the two copies at the corresponding marked points. 
The theorem is proved.
 \end{proof}

\end{document}